\documentclass{amsart}

\usepackage{amsfonts,amssymb,amsmath,amsthm,synttree}
\usepackage{relsize}
\usepackage[mathscr]{eucal}
\usepackage{url}
\usepackage{enumerate}
\usepackage{wrapfig}
\usepackage{caption}
\usepackage{subcaption}

\allowdisplaybreaks[1]

\PassOptionsToPackage{usenames,dvipsnames,svgnames}{xcolor}  
\usepackage{tikz}
\usetikzlibrary{arrows,positioning,automata}
\usetikzlibrary{shapes,snakes} 

\newtheorem{thrm}{Theorem}[section]
\newtheorem{lem}[thrm]{Lemma}

\newtheorem{cor}[thrm]{Corollary}
\theoremstyle{definition}
\newtheorem{definition}[thrm]{Definition}

\numberwithin{equation}{section}

\newcommand{\labeq}[1]{\label{eq:#1}}
\newcommand{\refeq}[1]{(\ref{eq:#1})}
\newcommand{\labt}[1]{\label{thm:#1}}
\newcommand{\reft}[1]{Theorem~\ref{thm:#1}}
\newcommand{\labl}[1]{\label{lemma:#1}}
\newcommand{\refl}[1]{Lemma~\ref{lemma:#1}}
\newcommand{\labd}[1]{\label{definition:#1}}
\newcommand{\refd}[1]{Definition~\ref{definition:#1}}
\newcommand{\labc}[1]{\label{coro:#1}}
\newcommand{\refc}[1]{Corollary~\ref{coro:#1}}

\newcommand{\labs}[1]{\label{sec:#1}}
\newcommand{\refs}[1]{Section~\ref{sec:#1}}
\newcommand{\labf}[1]{\label{fig:#1}}
\newcommand{\reff}[1]{Figure~\ref{fig:#1}}

\newcommand{\e}{\epsilon}

\newcommand{\formalsum}{\sum_{n=1}^{\infty} \frac {E_n} {q_1 q_2 \ldots q_n}}
\newcommand{\dimh}[1]{\hbox{dim$_{\hbox{H}}$}\left( #1\right)}
\newcommand{\dimb}[1]{\hbox{dim$_{\hbox{B}}$}\left( #1\right)}
\newcommand{\dimp}[1]{\hbox{dim$_{\hbox{P}}$}\left( #1\right)}

\newcommand{\lmeas}[1]{\lambda\left( #1 \right)}

\newcommand{\lpq}[1]{\mathscr{L}_{P,Q}\left( #1 \right)}
\newcommand{\vpq}[1]{\mathscr{V}_{P,Q}\left( #1 \right)}

\newcommand{\ri}[1]{\mathscr{R}_{\mathscr{I}}({ #1 })}

\newcommand{\LC}{\mathscr{C}_{P,Q}^{\hbox{L}}}
\newcommand{\RC}{\mathscr{C}_{P,Q}^{\hbox{R}}}
\newcommand{\LD}{\mathscr{D}_{P,Q}^{\hbox{L}}}
\newcommand{\RD}{\mathscr{D}_{P,Q}^{\hbox{R}}}

\newcommand{\BV}[1]{\hbox{BV}\left( #1 \right)}

\newcommand{\p}[1]{p_1 \cdots p_{ #1 }}
\newcommand{\q}[1]{q_1 \cdots q_{ #1 }}

\newcommand{\wrtP}{\hbox{ w.r.t. }P}

\newcommand{\NN}{\mathbb{N}_2^{\mathbb{N}}}
\newcommand{\NNC}{\NN \times \NN}
\newcommand{\PQ}{(P,Q)}
\newcommand{\measNN}{\mathscr{M} \left( \NN \right)}
\newcommand{\measNNz}{\mathscr{M} \left( \mathbb{N}_0^{\mathbb{N}} \right)}
\newcommand{\measwNN}{\mathscr{M}^w \left( \NN \right)}
\newcommand{\measeNN}{\mathscr{M}^e \left( \NN \right)}

\newcommand{\muonetwo}{\mu_1 \times \mu_2}
\newcommand{\pione}[1]{\pi_1\left( #1 \right)}
\newcommand{\pitwo}[1]{\pi_2\left( #1 \right)}
\newcommand{\ELone}{\int \log \pione{\omega} \ d\mu(\omega)}
\newcommand{\ELtwo}{\int \log \pitwo{\omega} \ d\mu(\omega)}
\newcommand{\ELonesq}{\int (\pione{\omega})^2 \ d\mu(\omega)}
\newcommand{\ELtwosq}{\int  (\pitwo{\omega})^2 \ d\mu(\omega)}

\newcommand{\ppq}{\psi_{P,Q}}
\newcommand{\pqp}{\psi_{Q,P}}
\newcommand{\ppqt}{\psi_{P_t,Q_t}}
\newcommand{\phpq}[1]{\phi_{P,Q}^{ (#1) }}
\newcommand{\phqp}[1]{\phi_{Q,P}^{ (#1) }}

\newcommand{\LS}[1]{\sim_{s_{ #1 }}}

\newcommand{\al}{\alpha}
\newcommand{\ga}{\gamma}
\newcommand{\be}{\beta}
\newcommand{\floor}[1]{\left\lfloor #1 \right\rfloor} 
\newcommand{\ceil}[1]{\left\lceil #1 \right\rceil} 

\newcommand{\NQ}{\mathscr{N}(Q)}
\newcommand{\N}[1]{\mathscr{N}( #1 )}
\newcommand{\Nk}[2]{\mathscr{N}_{#2}( #1 )} 
\newcommand{\DNQ}{\mathscr{DN}(Q)}
\newcommand{\DN}[1]{\mathscr{DN}( #1 )} 
\newcommand{\RNQ}{\mathscr{RN}(Q)}
\newcommand{\RN}[1]{\mathscr{RN}( #1 )} 
\newcommand{\RNk}[2]{\mathscr{RN}_{#2}( #1 )} 
\newcommand{\RDN}{\RNQ \cap \DNQ \backslash \NQ}

\newcommand{\NU}[1]{ \mathscr{NU}_{ #1 }} 
\newcommand{\U}[1]{ \mathscr{U}_{ #1}} 

\newcommand{\Nc}[1]{\mathbb{N}_{ #1 }}

\newcommand{\RNisect}{\bigcap_{j=1}^\infty \DN{Q_j} \backslash \RNk{Q_j}{1}}

\newcommand{\zpq}[1]{\mathscr{Z}_{P,Q}^{ (#1) }}
\newcommand{\zqp}[1]{\mathscr{Z}_{Q,P}^{ (#1) }}
\newcommand{\wpq}[1]{\mathscr{W}_{P}^{ (#1) }}
\newcommand{\wqp}[1]{\mathscr{W}_{Q}^{ (#1) }}

\newcommand{\EDP}{\mathbb{E}_D(P)}
\newcommand{\EDQ}{\mathbb{E}_{D,(t_n)}(Q)}

\author[B. Mance]{Bill Mance}
\address[B. Mance]{
\indent Department of Mathematics \newline
\indent University of North Texas \newline
\indent General Academics Building 435\newline
\indent 1155 Union Circle \#311430\newline
\indent Denton, TX 76203-5017, USA\newline
\indent Tel.: +1-940-369-7374\newline
\indent Fax: +1-940-565-4805}
\email{mance@unt.edu}

\begin{document}

\thanks{
Research of the author is partially supported by the U.S. NSF grant DMS-0943870.  Additionally, the author would like to thank Pieter Allaart, Michael Cotton, and Mariusz Urbanski for many helpful discussions. The author is indebted to the referee for many valuable suggestions that have improved this manuscript.}

\title[Applications of a class of fractal functions, I]{Number theoretic applications of a class of Cantor series fractal functions, I}

\begin{abstract}
Suppose that $(P,Q) \in \NNC$ and $x=E_0.E_1E_2\cdots$ is the $P$-Cantor series expansion of $x \in \mathbb{R}$.  We define
$\ppq(x):=\sum_{n=1}^\infty \frac {\min(E_n,q_n-1)} {\q{n}}$. 
The functions $\ppq$ are used to construct many pathological examples of normal numbers.  
These constructions are used to give the complete containment relation between the sets of $Q$-normal, $Q$-ratio normal, and $Q$-distribution normal numbers and their pairwise intersections for fully divergent $Q$ that are infinite in limit.  We analyze the H\"older continuity of $\ppq$ restricted to some judiciously chosen fractals.  This allows us to compute the Hausdorff dimension of some sets of numbers defined through restrictions on their Cantor series expansions.  In particular, the main theorem of a paper by Y. Wang {\it et al.} \cite{WangWenXi} is improved.

Properties of the functions $\ppq$ are also analyzed.  Multifractal analysis is given for a large class of these functions and continuity is fully characterized.  We also  study the behavior of $\ppq$ on both rational and irrational points, monotonicity, and bounded variation.  For different classes of ergodic shift invariant Borel probability measures $\mu_1$ and $\mu_2$ on $\NN$, we study which of these properties $\ppq$ satisfies for $\mu_1 \times \mu_2$-almost every $(P,Q) \in \NNC$. Related classes of random fractals are also studied.


\end{abstract}

\maketitle


\section{Introduction}\labs{intro}

The study of normal numbers and other statistical properties of real numbers with respect to large classes of Cantor series expansions was first studied by P. Erd\"{o}s and A. R\'{e}nyi.
This early work was done by P. Erd\"{o}s and A. R\'{e}nyi in \cite{ErdosRenyiConvergent} and \cite{ErdosRenyiFurther} and by A. R\'{e}nyi in \cite{RenyiProbability}, \cite{Renyi}, and \cite{RenyiSurvey}.  One of the main goals of this paper is to greatly expand upon their work and that of several other authors.  Applications to normal numbers are discussed in \refs{normal}.

The $Q$-Cantor series expansion, first studied by G. Cantor in \cite{Cantor},
\footnote{G. Cantor's motivation to study the Cantor series expansions was to extend the well known proof of the irrationality of the number $e=\sum 1/n!$ to a larger class of numbers.  Results along these lines may be found in the monograph of J. Galambos \cite{Galambos}.  See also \cite{TijdemanYuan} and \cite{HT}.  }
is a natural generalization of the $b$-ary expansion. Let $\mathbb{N}_k:=\mathbb{Z} \cap [k,\infty)$.  If $Q \in \NN$, then we say that $Q$ is a {\it basic sequence}.
Given a basic sequence $Q=(q_n)_{n=1}^{\infty}$, the {\it $Q$-Cantor series expansion} of a real $x$ in $\mathbb{R}$ is the (unique)\footnote{Uniqueness can be proven in the same way as for the $b$-ary expansions.} expansion of the form
\begin{equation} \labeq{cseries}
x=E_0+\sum_{n=1}^{\infty} \frac {E_n} {q_1 q_2 \ldots q_n}
\end{equation}
where $E_0=\floor{x}$ and $E_n$ is in $\{0,1,\ldots,q_n-1\}$ for $n\geq 1$ with $E_n \neq q_n-1$ infinitely often. We abbreviate \refeq{cseries} with the notation $x=E_0.E_1E_2E_3\ldots$ w.r.t. $Q$.

Clearly, the $b$-ary expansion is a special case of \refeq{cseries} where $q_n=b$ for all $n$.  If one thinks of a $b$-ary expansion as representing an outcome of repeatedly rolling a fair $b$-sided die, then a $Q$-Cantor series expansion may be thought of as representing an outcome of rolling a fair $q_1$ sided die, followed by a fair $q_2$ sided die and so on.

Let $x=E_0.E_1E_2\cdots$ w.r.t. $P$. If there are no values $n$ such that $E_n=0$ or $E_n=p_n-1$, then we let $\rho_P(x):=0$.  Otherwise, set $\rho_P(x):=\sup \{k \in \mathbb{N} : \exists n \in \mathbb{N}\hbox{ such that } E_{n+t} \in \{0,p_{n+t}-1\} \forall t \in [0,k-1]\}$. For $k \in \mathbb{N} \cup \{0,\infty\}$, put $\wpq{k}:=\{x \in \mathbb{R} : \rho_{P}(x) \leq k\}$ and 
$$
\zpq{k}:=\{x=0.E_1E_2\cdots \hbox{ w.r.t. }P : E_n<\min(p_n,q_n)\} \cap \wpq{k} \cap \left(\ppq\right)^{-1}\left(\wqp{k}\right).
$$
\begin{definition}\labd{ppq}
Let $(P,Q) \in \NNC$ and suppose that $x=E_0.E_1E_2 \cdots$ w.r.t. $P$.  We define
\footnote{We will use the symbol $:=$ only to define notation globally for the whole paper.}
$$
\ppq(x):=\sum_{n=1}^\infty \frac {\min(E_n,q_n-1)} {\q{n}} \hbox{ and }\phpq{k}:=\ppq \Big|_{\zpq{k}}.
$$
\end{definition}

The study of the functions $\ppq$ and $\phpq{k}$ and their applications to digital problems involving Cantor series expansions form the core of this paper.  Let $Q \in \NN$ and let $\NQ, \RNQ$, and $\DNQ$ be the sets of $Q$-normal numbers, $Q$-ratio normal numbers, and $Q$-distribution normal numbers, respectively.\footnote{We defer the definition of these sets to \refs{normal}.}
  The original motivation for the author to study the functions $\ppq$ was to study the set $\RDN$
\footnote{For a judiciously chosen $Q \in \NN$, we construct an explicit example of a member of $\RDN$ in \refs{RDN}.  It was previously unknown if there are any basic sequences $Q$ such that $\RDN \neq \emptyset$.}
and the sets constructed in the sequel to this paper by B. Li and the author \cite{ppq2}.  One of the more surprising applications of the methods introduced in this paper is that for every $k \geq 2$, there exists a basic sequence $Q$ and a real number $x$ that is $Q$-normal of order $k$, but not $Q$ normal of any order $1, 2, \cdots$, or $k-1$.  Explicit examples of computable basic sequences $Q$ and computable real numbers $x$ with this property are given in \cite{ppq2}.

The basic sequence $Q$ constructed in \refs{RDN} is a computable sequence and the member of $\RDN$ constructed in the same section is a computable real number.  No deep knowledge of computability theory will be used and any time we make such a claim there will exist a simple algorithm to compute the number under consideration to any degree of precision.
\refs{normal} is devoted to understanding the relationship between $\NQ, \RNQ$, and $\DNQ$ and intersections thereof.  
We refer to the directed graph in \reff{figure1} for the complete containment relationships between these notions when $Q$ is infinite in limit and fully divergent.  The vertices are labeled with all possible intersections of one, two, or three choices of the sets $\NQ$, $\RNQ$, and $\DNQ$.  The set labeled on vertex $A$ is a subset of the set labeled on vertex $B$ if and only if there is a directed path from $A$ to $B$.\footnote{The underlying undirected graph in \reff{figure1} has an isomorphic copy of complete bipartite graph $K_{3,3}$ as a subgraph.  Thus, it is not planar and the analogous directed graph that connects two vertices if and only if there is a containment relation between the two labels is more difficult to read.}  For example, $\NQ \cap \DNQ \subseteq \RNQ$, so all numbers that are $Q$-normal and $Q$-distribution normal are also $Q$-ratio normal.  A {\it block} is an ordered tuple of non-negative integers, a {\it block of length $k$} is an ordered $k$-tuple of integers, and {\it block of length $k$ in base $b$} is an ordered $k$-tuple of integers in $\{0,1,\ldots,b-1\}$.

The following is the main result of \refs{normal}.

\begin{thrm}
\reff{figure1} represents the complete containment relationship for basic sequences $Q$ that are infinite in limit and fully divergent.
\end{thrm}

\begin{figure}
\caption{}
\labf{figure1}
\begin{tikzpicture}[>=stealth',shorten >=1pt,node distance=3.8cm,on grid,initial/.style    ={}]
  \node[state]          (NQ)                        {$\mathsmaller{\NQ}$};
  \node[state]          (RNQ) [left =of NQ]    {$\mathsmaller{\RNQ}$};
  \node[state]          (NQRNQ) [above right =of NQ]    {$\mathsmaller{\NQ \cap \RNQ}$};
  \node[state]          (RNQDNQ) [above left=of RNQ]    {$\mathsmaller{\RNQ \cap \DNQ}$};
  \node[state]          (NQDNQ) [above left =of NQ]    {$\mathsmaller{\NQ \cap \DNQ}$};
  \node[state]          (NQRNQDNQ) [above right =of NQDNQ]    {$\mathsmaller{\NQ \cap \RNQ \cap \DNQ}$};
  \node[state]          (DNQ) [above right=of RNQDNQ]    {$\mathsmaller{\DNQ}$};
\tikzset{mystyle/.style={->,double=black}} 
\tikzset{every node/.style={fill=white}} 
\path (RNQDNQ)     edge [mystyle]    (RNQ)
      (RNQDNQ)     edge [mystyle]     (DNQ)
      (NQ)     edge [mystyle]     (RNQ)
      (NQDNQ)     edge [mystyle]     (RNQDNQ)
      (NQDNQ)     edge [mystyle]     (NQ);
\tikzset{mystyle/.style={<->,double=black}}
\path (NQRNQDNQ)     edge [mystyle]    (NQDNQ)
	(NQ)     edge [mystyle]    (NQRNQ);
\end{tikzpicture}
\end{figure}

Suppose that $M=(m_t)_t$ is an increasing sequence of positive integers.  Let $N_{M,n}^Q(B,x)$ be the number of occurrences of the block $B$ at positions $m_t$ for $m_t \leq n$ in the $Q$-Cantor series expansion of $\{x\}$.  For $m_t=t$ and $M=(m_t)$, let $N_n^Q(B,x) := N_{M,n}^Q(B,x)$.
We must also discuss the set of real numbers who have more than one expansion of the form \refeq{cseries} if we do not restrict $E_n<q_n-1$ infinitely often.  These are precisely the points $x=E_0.E_1E_2 \cdots E_n$ w.r.t. $Q$.  We note that if $x$ is of this form, then
$$
x=E_0+\sum_{j=1}^{n-1} \frac {E_j} {\q{j}} +\frac {E_n-1} {\q{n}}+\sum_{j=n+1}^\infty \frac {q_j-1} {\q{j}}.
$$
It should be noted that the distinction between these numbers will play a critical role in studying the properties of $\ppq$ as well as applications towards other problems.  Thus, for a basic sequence $Q$, we let $\mathscr{U}_Q:=\{x=E_0.E_1E_2\cdots\hbox{ w.r.t. }Q : E_n \neq 0 \hbox{ infinitely often}\}$ be the set of points with unique $Q$-Cantor series expansion and let $\mathscr{NU}_Q:=\mathbb{R} \backslash \mathscr{U}_Q$.\footnote{\refc{rationaldivides} gives conditions under which $\NU{Q}=\mathbb{Q}$.} 
The following theorem is not difficult to prove but will be of fundamental importance for the normal number constructions in this paper, the sequel to this paper with B. Li \cite{ppq2}, and those in planned future projects.

\begin{thrm}\labt{mainpsi}
\footnote{
The conclusions of \reft{mainpsi} sometimes do not hold without the requirement that $E_n<\min_{2 \leq r \leq j} (q_{r,n}-1)$ for infinitely many $n$.  For example, consider $p_n=3$ and 
\begin{displaymath}
q_n=\left\{ \begin{array}{ll}
2 & \textrm{if $n \equiv 0 \pmod{2}$}\\
3 & \textrm{if $n \equiv 1 \pmod{2}$}
\end{array} \right. .
\end{displaymath}
Let $x=7/8=0.\overline{21}$ w.r.t. $P$.  Then $\ppq(x)=1.\overline{0}$ w.r.t. $Q$ so $N_n^P((1),x)=\floor{n/2}$ while $N_n^Q((1),\ppq(x))=0$ for all $n$.
}
Suppose that $M=(m_t)$ is an increasing sequence of positive integers and  $Q_1=(q_{1,n}), Q_2=(q_{2,n}),\cdots, Q_j=(q_{j,n})$ are basic sequences and infinite in limit.  Set 
$$
\Psi_j(x)=\left(\psi_{Q_{j-1},Q_j} \circ \psi_{Q_{j-2},Q_{j-1}} \circ \cdots \circ \psi_{Q_1,Q_2}\right)(x).
$$
  If $x=E_0.E_1E_2\cdots$ w.r.t $Q_1$ satisfies
$E_n<\min_{2 \leq r \leq j} (q_{r,n}-1)$ for infinitely many $n$, then $\Psi_j(x) \in \U{Q_j}$ and for every block~$B$
$$
N_{M,n}^{Q_j}\left(B,\Psi_j(x)\right)
=N_{M,n}^{Q_1}(B,x)+O(1).
$$
\end{thrm}

The functions $\ppq$ and $\phpq{k}$ are interesting in their own right.  There is a vast literature studying functions with pathological properties.  An early example due to Weierstrauss is of a class of continuous and nowhere differentiable functions.  The study of other functions such as the Cantor function, Minwoski's question mark function, and the Takagi function also provides motivation for \refs{ppq}.  We give only a few references as relevant starting points: \cite{TakagiSurvey},  \cite{QuestionOne}, and \cite{Hardywire}.  We also mention that other fractal functions defined through Cantor series have been studied by H. Wang and Z. Xu in \cite{WangXu1} and \cite{WangXu2}.  However, these functions are quite different from the $\ppq$ and $\phpq{k}$ functions we study in this paper.

For a set $S \subseteq \mathbb{R}$, we will let $\lambda(S)$ denote the Lebesgue measure of $S$ and $\dimh{S}, \dimp{S}$, and $\dimb{S}$ will denote the Hausdorff, packing, and box dimensions of $S$, respectively.
In \refs{ppq}, we will examine many properties of the functions $\ppq$ including, but not limited to, rationality, continuity, and bounded variation. 
We will also study the level sets of $\ppq$ and multifractal analysis of $\ppq$. For simplicity, we will only consider the level sets of $\ppq$ in $(0,1]$ as  $\ppq$ is $1$-periodic and $\ppq(x)=0$ if and only if $x \in \mathbb{Z}$.
For $w \in (0,1]$, put 

$$
\lpq{w}:=\{x \in (0,1) : \ppq(x)=w\}.
$$
 For $\alpha \in [0,1]$, let
\begin{align*}
\vpq{\al}:=\{ w \in \ppq((0,1)) : \dimh{\lpq{w}}=\al\}
\end{align*}
be a level set of the function $\dimh{\lpq{\cdot}}$.  Let $\tau_n=n(n+1)/2$ be the $n$'th triangular number.  An eventually non-decreasing  sequence of real numbers $(s_n)$ {\it grows nicely}
\footnote{
Note that if $(s_n)$ grows nicely, then $\lim_{n \to \infty} s_n=\infty$.}
 if
$$
\lim_{n \to \infty} \frac {\log s_{\tau_{n+2}}} {\log s_{\tau_n}}=1.
$$
We will prove the following theorem.
\begin{thrm}\labt{multifractal}
For $(P,Q) \in \NNC$, let $r_n=p_n-q_n$.  If $(p_n), (q_n)$, and $(r_n)$ grow nicely, $p_n>q_n$ for all natural numbers $n$, and
$$
\lim_{n \to \infty} \frac {\log r_n} {\log p_n} = \gamma \in (0,1],
$$
then for all $\al \in [0,1]$
$$
\dimh{\vpq{\al}} \geq 1-\frac {\al} {\gamma}.
$$
Thus, $\dimh{\vpq{\al}}>0$ if $0 \leq \al<\gamma$.
\footnote{
The conditions of \reft{multifractal} are not very restrictive.  Most monotone sequences $(p_n)$ and $(q_n)$ that do not grow unreasonably fast and where $(p_n)$ dominates $(q_n)$ will satisfy the conditions of \reft{multifractal}.  For example, $p_n=2^n$ and $q_n=n+1$ satisfy this condition for $\gamma=1$.  A graph of $\ppq$ for these choices of $P$ and $Q$ is given in \reff{a}.  If $p_n=n^2+n$ and $q_n=n^2+1$, then the hypotheses of \reft{multifractal} are satisfied with $\gamma=1/2$.}
\end{thrm}

 While some properties such as continuity may easily be described for arbitrary choices of $P$ and $Q$, others will be too difficult to analyze for completely arbitrary choices.  Thus, for certain classes of ergodic and shift-invariant Borel probability measures $\mu_1, \mu_2$ on $\NN$ we will study these properties for $\mu_1 \times \mu_2$-almost every $\PQ \in \NNC$.  This will naturally give rise to many random fractals that we will consider.  We also include graphs of $\ppq$ for many choices of $P$ and $Q$ in \reff{graphs}.

The H\"older and Lipschitz continuity of $\phpq{k}$ is explored in \refs{Holder}.  This allows us to compute the Hausdorff dimension of some fractals defined through digital restrictions of Cantor series expansions in \refs{dimh}.  Additionally, we will use the results of \refs{Holder} to improve the main theorem in the paper \cite{WangWenXi} by Y. Wang {\it et al} and a result of the author in \cite{Mance7}.
For the remainder of this paper, we will assume the convention that the empty sum is equal to $0$ and the empty product is equal to $1$.  
\section{The functions $\ppq$ and $\phpq{k}$}\labs{ppq}

\begin{figure}
        \centering
        \begin{subfigure}[b]{0.3\textwidth}
                \includegraphics[width=\textwidth]{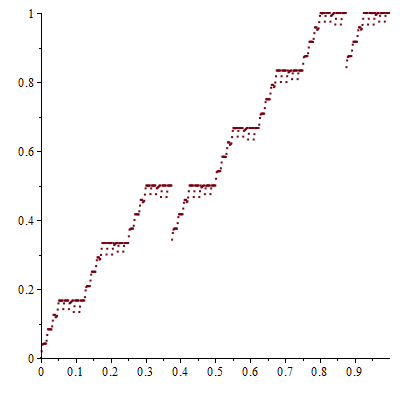}
                \caption{$\substack{p_n=2^n \\ q_n=n+1}$}
                \labf{a}
        \end{subfigure}%
        ~ 
        \begin{subfigure}[b]{0.3\textwidth}
                \includegraphics[width=\textwidth]{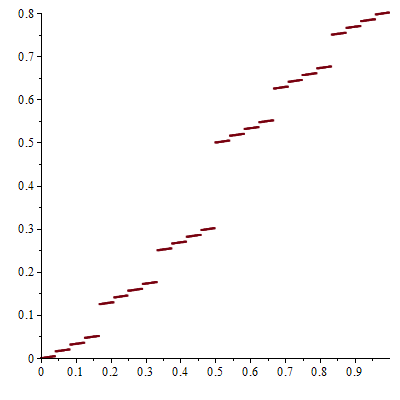}
                \caption{$\substack{p_n=n+1 \\ q_n=2^n}$}
                \labf{b}
        \end{subfigure}
        ~ 
         \begin{subfigure}[b]{0.3\textwidth}
                \includegraphics[width=\textwidth]{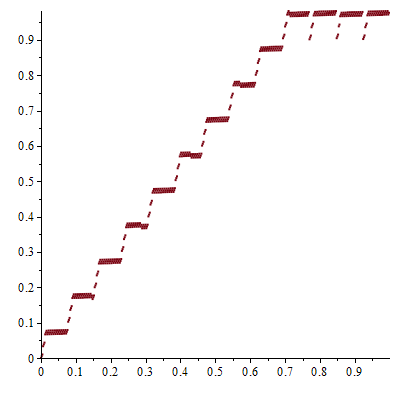}
                \caption{$\substack{P=(13, 13, 7, 9, 2, 7, 2, 7, 5, 17, \cdots) \\ Q=(10, 3, 15, 12, 15, 6, 7, 9, 6, 17, \cdots)}$}
                \labf{c}
        \end{subfigure}
        \begin{subfigure}[b]{0.3\textwidth}
                \includegraphics[width=\textwidth]{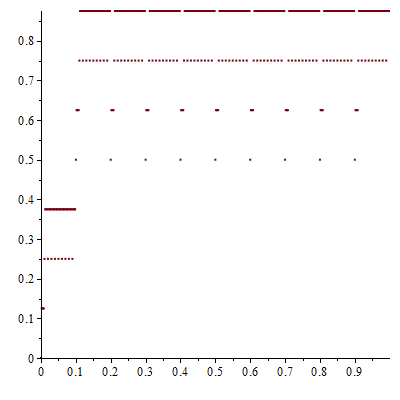}
                \caption{$\substack{p_n=10 \\ q_n=2}$}
                \labf{d}
        \end{subfigure}
        \begin{subfigure}[b]{0.3\textwidth}
                \includegraphics[width=\textwidth]{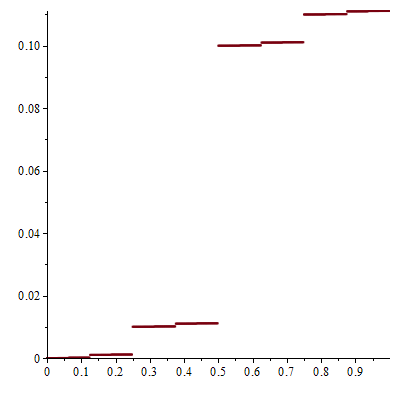}
                \caption{$\substack{p_n=2 \\ q_n=10}$}
                \labf{e}
        \end{subfigure}
        \begin{subfigure}[b]{0.3\textwidth}
                \includegraphics[width=\textwidth]{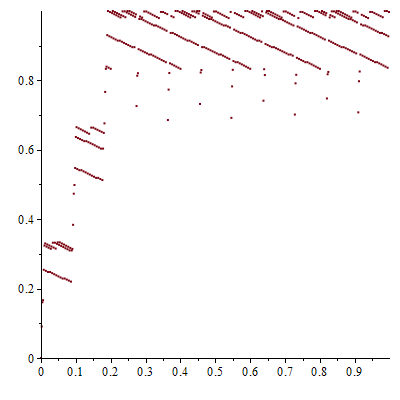}
                \caption{$\substack{P=(11,15, 13, 8, 13, 14, 2, 8, 4, 5,\cdots) \\ Q=(3, 2, 8, 7, 9, 2, 11, 8, 11, 12,\cdots)}$}
                \labf{f}
        \end{subfigure}
        \begin{subfigure}[b]{0.3\textwidth}
                \includegraphics[width=\textwidth]{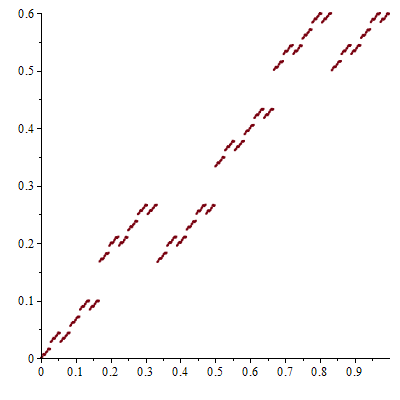}
                \caption{$\substack{P=(2, 3, 2, 3, 2, 3, 2, 3, 2, 3, \cdots) \\ Q=(3, 2, 3, 2, 3, 2, 3, 2, 3, 2,\cdots)}$}
                \labf{g}
        \end{subfigure}
        \begin{subfigure}[b]{0.3\textwidth}
                \includegraphics[width=\textwidth]{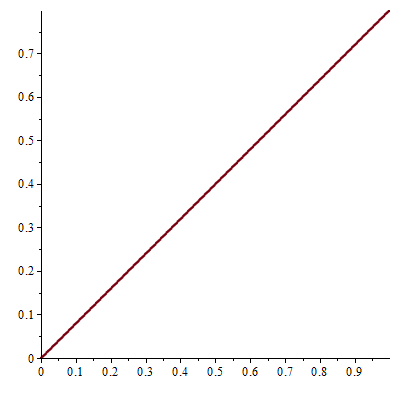}
                \caption{$\substack{P=(64, 73, 9, 173, 169, 61, 179, 43, 182, 108,\cdots) \\ Q=(80, 83, 10, 190, 151, 11, 101, 132, 41, 77,\cdots)}$}
                \labf{h}
        \end{subfigure}
        \begin{subfigure}[b]{0.3\textwidth}
                \includegraphics[width=\textwidth]{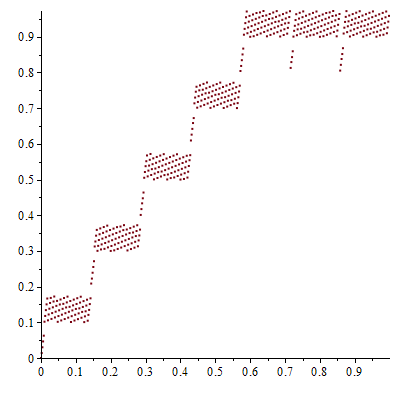}
                \caption{$\substack{P=(7, 15, 9, 2, 7, 3, 4, 9, 4, 4,\cdots) \\ Q=(5, 2, 12, 2, 15, 15, 12, 8, 8, 12,\cdots)}$}
                \labf{i}
        \end{subfigure}
        \begin{subfigure}[b]{0.3\textwidth}
                \includegraphics[width=\textwidth]{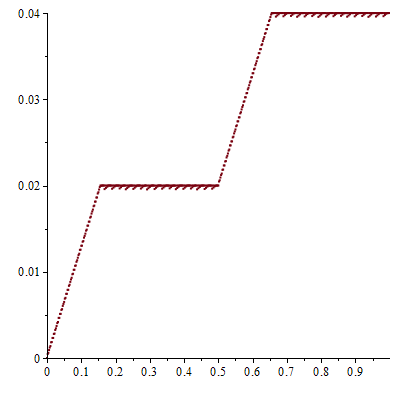}
                \caption{$\substack{P=(2, 137, 103, 87, 24, 143, 54, 170, 100, 182) \\ Q=(50, 43, 33, 48, 50, 15, 92, 164, 23, 33)}$}
                \labf{j}
        \end{subfigure}
        \begin{subfigure}[b]{0.3\textwidth}
                \includegraphics[width=\textwidth]{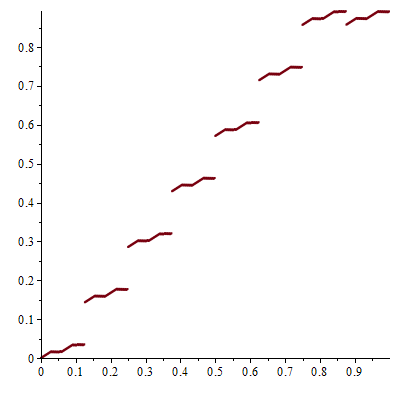}
                \caption{$\substack{P=(8,2,10,10,9,9,3,10,6,5,\cdots) \\ Q=(7,8,5,3,7,5,10,4,10,9,\cdots)}$}
                \labf{k}
        \end{subfigure}
        \begin{subfigure}[b]{0.3\textwidth}
                \includegraphics[width=\textwidth]{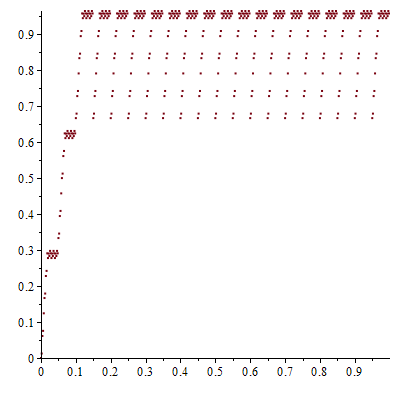}
                \caption{$\substack{P=(20, 15, 3, 19, 7, 19, 19, 19, 11, 5,\cdots) \\ Q=(3, 6, 8, 4, 11, 3, 9, 15, 17, 8,\cdots)}$}
                \labf{l}
        \end{subfigure}
        \caption{Graphs of $\ppq$ for different choices of $P$ and $Q$ plotted with $500$ pixels each.  Most graphs without an explicit formula for $p_n$ and $q_n$ were generated randomly.}\labf{graphs}
\end{figure}

For $Q \in \NN$ and a sequence of natural numbers $(a_j)$, define 
$$
\mathscr{R}_{(a_j)}(Q):=\{x=0.E_1E_2\cdots\hbox{ w.r.t. }Q : E_j<a_j\}.
$$
We note the following result due to H. Wegmann in \cite{Wegmann}:
\begin{thrm}\labt{wegmannsalat}
If $Q=(q_n) \in \NN$ and $\lim_{n \to \infty} \frac {\log q_n} {\log \q{n}}=0$, then
$$
\dimh{\mathscr{R}_{(a_j)}(Q)}=\liminf_{n \to \infty} \frac {\log \prod_{j=1}^n \min(a_j,q_j)} {\log \prod_{j=1}^n q_j}.
$$
\end{thrm}

The next theorem directly follows from \refd{ppq} and \reft{wegmannsalat}.

\begin{thrm}\labt{DRdim} 
If $(P,Q) \in \NNC$, then
$$
\ppq(\mathbb{R})=\mathscr{R}_{(\min(p_j,q_j))}(Q) \subseteq \left[0,\sum_{n=1}^\infty \frac {\min(p_n-1,q_n-1)} {\q{n}}\right] \hbox{ and }
\lmeas{\ppq(\mathbb{R})}=\prod_{j=1}^{\infty} \frac {\min(p_j,q_j)} {q_j}.
$$
Moreover, if $\lim_{n \to \infty} \frac {\log q_n} {\log \q{n}}=0$, then
$$
\dimh{\ppq(\mathbb{R})}=\liminf_{n \to \infty} \frac {\log \prod_{j=1}^n \min(p_j,q_j)} {\log \prod_{j=1}^n q_j}.
$$
\end{thrm}
Thus, the range of $\ppq$ can be anywhere from the interval $[0,1]$ to a Cantor set.  Given $Q \in \NN$, let $\mathscr{I}=(I_n)$, where $I_n \subseteq \{0,1,\cdots,q_n-1\}$.  
For the rest of this paper, define $\mathscr{R}_{\mathscr{I}}(Q)$ by
$$
\mathscr{R}_{\mathscr{I}}(Q):=\left\{x=\formalsum : E_j \in I_j\right\}.
$$
The proof of \reft{wegmannsalat} presented in \cite{Wegmann} can trivially be modified to arrive at the following generalization of \reft{wegmannsalat} that will frequently be used in this paper.

\begin{thrm}\labt{morewegmann}
Suppose that $Q=(q_n) \in \NN, \lim_{n \to \infty} \frac {\log q_n} {\log \q{n}}=0, I_j \subseteq \{0,1,\cdots,q_j-1\}$, and $\mathscr{I}=(I_n)$. Then
$$
\dimh{\mathscr{R}_{\mathscr{I}}(Q)}=\liminf_{n \to \infty} \frac {\log \prod_{j=1}^n | I_j|} {\log \prod_{j=1}^n q_j}.
$$
\end{thrm}
It should be noted that the sets $\mathscr{R}_{\mathscr{I}}(Q)$ are homogenous Moran sets.  Using corollary 3.1 from Feng {\it et al.} \cite{FengWenWu}, we have the following result connecting the Hausdorff, packing, and box dimensions of $\mathscr{R}_{\mathscr{I}}(Q)$.
\begin{lem}\labl{dimsame}
If 
\begin{equation}\labeq{dimsame}
\liminf_{n \to \infty} \frac {\log \prod_{j=1}^n |I_j|} {\log \prod_{j=1}^{n+1} |q_j|-\log |I_{n+1}|}=
\limsup_{n \to \infty} \frac {\log \prod_{j=1}^{n+1} |I_j|} {\log \prod_{j=1}^{n} |q_j|+\log |I_{n+1}|},
\end{equation}
then $\dimh{\ri{Q}}=\dimp{\ri{Q}}=\dimb{\ri{Q}}$.
\end{lem}

Lastly, we give a proof of \reft{mainpsi}.
\begin{proof}[Proof of \reft{mainpsi}]
Let $B=(b_1,b_2,\cdots,b_k)$. 
We use induction on $j$.  The base case $j=1$ is trivial.  Suppose now that $j \geq 2$  and
$N_{M,n}^{Q_{j-1}}\left(B,\Psi_{j-1}(x)\right)
=N_{M,n}^{Q_1}(B,x)+O(1).$
Put $b=\max (b_1,b_2,\cdots,b_k)$ and let $\Psi_{j-1}(x)=F_0.F_1F_2\cdots$ w.r.t. $Q_{j-1}$ and $\Psi_{j}(x)=G_0.G_1G_2\cdots$ w.r.t. $Q_j$.
Since $\min(E_n,q_2-1,\cdots,q_{j-1}-1) \leq E_n < q_{n,j}-1$ for infinitely many $n$, we know that $\Psi_j(x) \in \U{Q_j}$.  Let $t$ be large enough that $b<\min_{1 \leq r \leq j} (q_{j,n}-1)$ for every $n \geq t$.  Since $\Psi_j(x) \in \U{Q_j}$, we know for $n \geq t$ that $G_n \in \{0,1,\cdots,b\}$ if and only if $F_n  \in \{0,1,\cdots,b\}$.  Thus, 
\begin{align*}
&N_{M,n}^{Q_{j-1}}\left(B,\Psi_{j-1}(x)\right)-N_{M,t}^{Q_{j-1}}\left(B,\Psi_{j-1}(x)\right) 
\leq N_{M,n}^{Q_j}\left(B,\Psi_j(x)\right) \\
\leq &\left(N_{M,n}^{Q_{j-1}}\left(B,\Psi_{j-1}(x)\right)-N_{M,t}^{Q_{j-1}}\left(B,\Psi_{j-1}(x)\right)\right)+t,
\end{align*}
so $N_{M,n}^{Q_j}\left(B,\Psi_j(x)\right)=N_{M,n}^{Q_1}(B,x)+O(1)$
and \reft{mainpsi} is proven.
\end{proof}

\subsection{Level Sets and Multifractal Analysis of $\ppq$}\labs{level}
We wish to examine the range of $\ppq$ beyond what was discussed in \reft{DRdim}.  Our main tool will be \reft{morewegmann}.  For this subsection, we will assume that $\lim_{n \to \infty} \frac {p_n} {\p{n}}=0$ so that we may use \reft{morewegmann}.  
We will see in \refs{monotone} that the level sets $\lpq{w}$ are always empty, a single point, or a totally disconnected set. 


The next theorem follows directly from the definition of the Cantor series expansions and $\ppq$ and gives a complete characterization of the level sets of $\ppq$.  None of the following statements are difficult to prove so we omit their proofs.
\begin{thrm}\labt{level}
Suppose that $w \in (0,1]$ and $x \in \lpq{w}$.  We write $w=E_0.E_1E_2\cdots$ w.r.t. $Q$ and $x=0.F_1F_2\cdots \wrtP$.  
\begin{enumerate}
\item If $E_n\in [0,q_n-2]$ and there exists $m>n$ such that $E_m \neq 0$, then $F_n=E_n$.
\item If $E_n =q_n-1$ and there exists $m>n$ such that $E_m \neq 0$, then $F_n \in [q_n-1,p_n-1]$.
\item If $w \in \NU{Q} \cap (0,1)$ and $n=\sup\{t \in \mathbb{N} : E_t>0\}$, then $\lpq{w}=A \cup B$, where 
$$
A=\{\zeta=0.G_1\cdots G_n \hbox{ w.r.t }P : \ppq(\zeta)=w\} \hbox{ and}
$$
\begin{align*}
B=\Bigg\{0.G_1\cdots G_{n-1} (E_n-1)G_{n+1}\cdots \hbox{ w.r.t }P:\ & G_m \in [q_m-1,p_m-1] \forall m>n\\ 
& \hbox{ and } \ppq \left(\sum_{j=1}^{n-1} \frac {G_j} {\p{j}} \right)=\sum_{j=1}^{n-1} \frac {E_j}{\q{j}}\Bigg\}.
\end{align*}
Clearly, the set $A$ is at most finite although the set $B$ may be quite large.  Also
$$
\lpq{1}=\left\{0.G_1G_2\cdots \wrtP: G_m \in [q_m-1,p_m-1] \forall m \in \mathbb{N} \right\}.
$$
\item If $w \in \U{Q}$, then $\lpq{w}=\emptyset$ if and only if there exists a natural number $n$ such that  $E_n \geq p_n$.
\item If there exists $n$ with $p_n<q_n$, then $\lpq{w} =~\emptyset$ for all $w > \sum_{n=1}^\infty \frac {\min(p_n-1,q_n-1)}{\q{n}}$.
\item If $p_n > q_n$ for at most finitely many $n$, then $\mathscr{L}_{P,Q}(w)$ is finite for all $w \in \ppq((0,1))$. 
\item If $p_n \leq q_n$ for all $n$, then $\ppq$ is injective and increasing.
\end{enumerate}
\end{thrm}

For $w=0.E_1E_2\cdots$ w.r.t. $Q$, set
\begin{displaymath}
\omega_n(w)=\left\{ \begin{array}{ll}
1 & \textrm{if $E_n \in [0,q_n-2]$}\\
p_n-q_n+1 & \textrm{if $E_n=q_n-1$}\\
0 & \textrm{if $E_n\geq p_n$}\\
\end{array} \right. .
\end{displaymath}

\begin{thrm}\labt{dimhlevel}
Let $w=0.E_1E_2\cdots$ w.r.t. $Q$. 
If $w \in \U{Q}$, then
\begin{equation}\labeq{lpqmeashd}
\lmeas{\lpq{w}}=\prod_{j=1}^\infty \frac {\omega_j(w)} {p_j} \hbox{ and }
\dimh{\lpq{w}}=\liminf_{n \to \infty} \frac {\log \prod_{j=1}^n \omega_j(w)} {\log \prod_{j=1}^n p_j}.
\end{equation}
If $w \in \NU{Q}$, $M=\inf\{t \in \mathbb{N} : E_t>0\}$, and $p_n \geq q_n$ for all $n >M$, then
\begin{align*}\labeq{lpqmeashd2}
\lmeas{\lpq{w}}&=\left(\prod_{j=1}^{M-1} \frac {\omega_j(w)} {p_j}\right) \cdot\frac {1} {p_M} \cdot \left(\prod_{j=M+1}^\infty \frac {p_j-q_j+1} {p_j}\right) \hbox{ and }\\
\dimh{\lpq{w}}&=\liminf_{n \to \infty} \frac {\log \prod_{j=M+1}^n \omega_j(w)} {\log \prod_{j=M+1}^n p_j}.
\end{align*}
If $\lpq{w} \neq \emptyset$ and \refeq{dimsame} holds with
\begin{displaymath}
I_n=I_n(w)=\left\{ \begin{array}{ll}
\{E_n\} & \textrm{if $E_n \in [0,q_n-2]$}\\
\left[q_n-1, p_n-1\right] & \textrm{if $E_n=q_n-1$}
\end{array} \right. ,
\end{displaymath}
 then $\dimh{\lpq{w}}=\dimp{\lpq{w}}=\dimb{\lpq{w}}$.
\end{thrm}
\begin{proof}
This follows from  \reft{morewegmann} and our characterization of the level sets of $\ppq$ in \reft{level}.  The last part follows from \refl{dimsame}.
\end{proof}

We will need the following basic lemmas to help prove \reft{multifractal}.

\begin{lem}\labl{tcorr}
Let $L$ be a real number and $(a_n)_{n=1}^\infty$ and $(b_n)_{n=1}^\infty$ be two sequences of positive real numbers such that
$$
\sum_{n=1}^{\infty} b_n=\infty \hbox{ and } \lim_{n \to \infty} \frac {a_n} {b_n}=L.
$$
Then
$$
\lim_{n \to \infty} \frac {a_1+a_2+\ldots+a_n} {b_1+b_2+\ldots+b_n}=L.
$$
\end{lem}

\begin{lem}\labl{tcorr2}
Let $L$ be a real number and $(a_n)_{n=1}^\infty$ and $(b_n)_{n=1}^\infty$ be two sequences of positive integers.  Let $(c_t)_{t=1}^\infty$ be an increasing sequence of positive integers.  Set $A_t=\sum_{n=c_t}^{c_{t+1}-1} a_n$ and $B_t=\sum_{n=c_t}^{c_{t+1}-1} b_n$.  If $\lim_{t \to \infty} \frac {A_1+A_2+\cdots+A_t} {B_1+B_2+\cdots+B_t}=L,\  \sum_{t=1}^\infty A_t=\sum_{t=1}^\infty B_t=\infty,$ and $\lim_{t \to \infty} \frac {A_{t+1}} {A_1+A_2+\cdots+A_t}=\lim_{t \to \infty} \frac {B_{t+1}} {B_1+B_2+\cdots+B_t}=0$, then
$$
\lim_{n \to \infty} \frac {a_1+a_2+\ldots+a_n} {b_1+b_2+\ldots+b_n}=L.
$$
\end{lem}

We also need
\begin{lem}\labl{nice}
If $(s_n)$ grows nicely, then $\lim_{n \to \infty} \frac {\log s_n} {\log s_1s_2 \cdots s_n}=0$.
\end{lem}
\begin{proof}
Let $m \in \mathbb{N}$.  For $t>m$ and $n \in [\tau_{t+1}  ,\tau_{t+2})$
$$
\frac {\log s_n} {\log s_1s_2 \cdots s_n} \leq \frac {\log s_{\tau_{t+2}}} {\log \prod_{j=\tau_t}^{\tau_t+m} s_j}.
$$
Since $(s_n)$ is monotone
$$
\lim_{n \to \infty} \frac {\log s_n} {\log s_1s_2 \cdots s_n} \leq \lim_{t \to \infty} \frac {\log s_{\tau_{t+2}}} {\sum_{j=\tau_t}^{\tau_t+m} \log  s_j}=\frac {1} {m}
$$
and $\lim_{n \to \infty} \frac {\log s_n} {\log s_1s_2 \cdots s_n}=0$.
\end{proof}

\begin{proof}[Proof of \reft{multifractal}]

Let $\al < \gamma$ and $N$ be the smallest integer such that $p_n>q_n$, $q_n>2$, and $(p_n), (q_n)$, and $(r_n)$ are non-decreasing for all $n>N$.
We will describe a set $S\subseteq (0,1]$ where $\dimh{\lpq{w}}=\al$ for all $w \in S$ and $\dimh{S} = 1-\frac {\al} {\gamma}$.  Let
$$
c_t=\sum_{j=1}^t \ceil{\left(1-\frac {\al} {\gamma}\right)j}+\ceil{\frac {\al} {\gamma} j}.
$$
Set $M=\min \{t : c_t>N\}$,
\begin{displaymath}
I_n=\left\{ \begin{array}{ll}
\{0,1,\cdots,q_n-2\} & \textrm{if $n \in \left[c_t,c_t+\ceil{\left(1-\frac {\al} {\gamma}\right)t}-1\right]$}\\
 & \\
\{q_n-1\} & \textrm{if $n \in \left[c_t+\ceil{\left(1-\frac {\al} {\gamma}\right)t},c_{t+1}-1\right]$}\\
\end{array} \right.,
\end{displaymath}
$\mathscr{I}=(I_n)$, and $S=\mathscr{R}_{\mathscr{I}}(Q)$.
Let $w \in S$ and put 
\begin{align*}
A_t&=\ceil{\frac {\al} {\gamma} t} \log (r_{c_t}), &A_t'=\sum_{n=c_t}^{c_{t+1}-1} \log \omega_n(w)\\
B_t&=\left(\ceil{\left(1-\frac {\al} {\gamma}\right)t}+\ceil{\frac {\al} {\gamma} t}\right) \log (p_{c_t}), &B_t'=\sum_{n=c_t}^{c_{t+1}-1} \log p_n.
\end{align*}
Thus, $\lim_{t \to \infty} \frac {A'_t} {A_t}=\lim_{n \to \infty} \frac {B'_t} {B_t}=1$ since $(r_n)$ and $(p_n)$ are nice sequences.  Since $(r_n)$ is eventually monotone and $r_n \to \infty$
\begin{align*}
&\lim_{t \to \infty} \frac {A_{t+1}} {A_M+\cdots+A_t}=\lim_{t \to \infty} \frac {\ceil{\frac {\al} {\gamma}(t+1)} \log r_{c_{t+1}}} {\sum_{j=M}^t \ {\ceil{\frac {\al} {\gamma}j} \log r_{c_j}}}
\leq \lim_{t \to \infty} \frac {\ceil{\frac {\al} {\gamma}(t+1)} \log r_{c_{t+1}}} {\sum_{j=M}^t \ {\ceil{\frac {\al} {\gamma}j} \log r_{c_{t}}}}\\
&=\lim_{t \to \infty} \frac {\ceil{\frac {\al} {\gamma}(t+1)}} {\sum_{j=M}^t\ceil{\frac {\al} {\gamma}j}} = \lim_{t \to \infty} \frac {t+1} {t(t+1)/2+O(t)}=0.
\end{align*}
Similarly, it can be shown that 
$\lim_{t \to \infty} \frac {B_{t+1}} {B_M+\cdots+B_t}=0$, so by \refl{tcorr} and \refl{tcorr2}
\begin{equation}\labeq{tcorruse}
\lim_{n \to \infty} \frac {\sum_{j=c_M}^n \log \omega_j(w)} {\sum_{j=c_M}^n \log p_j}=\lim_{t \to \infty} \frac {A_t} {B_t}
=\lim_{t \to \infty} \left(\frac {\ceil{\frac {\al} {\gamma} t}} {\ceil{\left(1-\frac {\al} {\gamma}\right)t}+\ceil{\frac {\al} {\gamma} t}} \cdot \frac {\log r_{c_t}} {\log p_{c_t}}\right) = \frac {\al} {\gamma} \cdot \gamma=\alpha.
\end{equation}
By construction, $w \in \U{Q}$.  Thus, by \refeq{lpqmeashd}, \refeq{tcorruse}, and \refl{nice}
$$
\dimh{\lpq{w}}=\liminf_{n \to \infty} \frac {\log \prod_{j=c_M}^n \omega_j(w)} {\log \prod_{j=c_M}^n p_j}=\al.
$$

Let
\begin{displaymath}
\upsilon_n=|I_n|=\left\{ \begin{array}{ll}
q_n-1 & \textrm{if $n \in \left[c_t,c_t+\ceil{\left(1-\frac {\al} {\gamma}\right)t}-1\right]$}\\
 & \\
1 & \textrm{if $n \in \left[c_t+\ceil{\left(1-\frac {\al} {\gamma}\right)t},c_{t+1}-1\right]$}\\
\end{array} \right. .
\end{displaymath}
Then, by \reft{morewegmann} and \refl{nice}
$$
\dimh{S}=\liminf_{n \to \infty} \frac {\log \prod_{j=c_M}^n \upsilon_j(w)} {\log \prod_{j=c_M}^n q_j}
$$
A similar argument using  \refl{tcorr} and \refl{tcorr2} shows that $\dimh{S}=1-\frac {\al} {\gamma}$.  Thus, since $S \subseteq \vpq{\al}$, we know that $\dimh{\vpq{\al}} \geq 1-\frac {\al} {\gamma}$ and $\dimh{\vpq{\al}}>0$ if $\al<\gamma$.
\end{proof}

\begin{thrm}\labt{summeaslpq}
\footnote{
The proof of \reft{summeaslpq} can be modified to give a formula for $\sum_{w \in \ppq((0,1))} \lmeas{\lpq{w}}$ when $p_n<q_n$ at most finitely often.  For clarity, we have only presented the case where $p_n \geq q_n$ for all $n$.}
Suppose that  $p_n \geq q_n$ for all $n$ and $\sum \frac {q_n} {p_n}<\infty$.  Then 
$\lmeas{\lpq{w}}>0$
 if and only if $w \in \NU{Q} \cap \ppq((0,1))$.  
Furthermore,
$$
\sum_{w \in \ppq((0,1))} \lmeas{\lpq{w}}=\sum_{w \in \NU{Q} \cap \ppq((0,1))} \lmeas{\lpq{w}}
=\sum_{k=1}^\infty \frac {q_k-1} {p_k} \cdot \prod_{j=k+1}^\infty \left(1-\frac {q_j-1} {p_j}\right).
$$
\end{thrm}
\begin{proof}
We first note that $\sum \frac{ q_n} {p_n}$ converges if and only if $\sum \frac {q_n-1} {p_n}$ converges.  An argument that shows this is given in the proof of \reft{ppqrat}.  Let $M=\inf\{t \in \mathbb{N} : E_t>0\}$.  Then by \refeq{lpqmeashd} for $w=0.E_1E_2\cdots E_M$ w.r.t.
$$
\lmeas{\lpq{w}}=\left(\prod_{j=1}^{M-1} \frac {\omega_j(w)} {p_j}\right) \cdot\frac {1} {p_M} \cdot \left(\prod_{j=M+1}^\infty \frac {p_j-q_j+1} {p_j}\right).
$$
Since $\sum \frac {q_n-1} {p_n}$ converges,
$$
\prod_{j=M+1}^\infty \frac {p_j-q_j+1} {p_j}=\prod_{j=M+1}^\infty \left(1-\frac {q_j-1} {p_j}\right)>0,
$$
so $\lmeas{\lpq{w}}>0$.
If $w \in \U{Q}$, then
$$
\lmeas{\lpq{w}} \leq \prod_{\substack{1 \leq j <\infty\\ E_j \neq q_j-1}} \frac {1} {p_j}=0,
$$
by \refeq{lpqmeashd}, so $\lmeas{\lpq{w}}=0$. 

We will now evaluate $\sum_{w \in \NU{Q} \cap \ppq((0,1))} \lmeas{\lpq{w}}$.  Let
\begin{displaymath}
\xi_n(m)=\left\{ \begin{array}{ll}
1 & \textrm{if $m \in [0,q_n-2]$}\\
p_n-q_n+1 & \textrm{if $m=q_n-1$}\\
0 & \textrm{if $m\geq p_n$}\\
\end{array} \right.
\end{displaymath}
and put $\Upsilon_k=\prod_{j=k+1}^\infty \frac {p_j-q_j+1} {p_j}=\prod_{j=k+1}^\infty \left(1-\frac {q_j-1} {p_j}\right)>0$.
Then by \refeq{lpqmeashd}
\begin{align*}
&\sum_{w \in \NU{Q} \cap \ppq((0,1))} \lmeas{\lpq{w}}
=\sum_{k=1}^\infty \sum_{\substack{0 \leq E_1 \leq q_1-1\\ \cdots \\ 0 \leq E_{k-1} \leq q_{k-1}-1 \\ 1 \leq E_k \leq q_k-1}} \lmeas{\lpq{\sum_{n=1}^k \frac {E_n} {\q{n}}}}\\
=&\sum_{k=1}^\infty \sum_{\substack{0 \leq E_1 \leq q_1-1\\ \cdots \\ 0 \leq E_{k-1} \leq q_{k-1}-1 \\ 1 \leq E_k \leq q_k-1}} \left(\prod_{j=1}^{k-1} \frac {\xi_j(E_j)} {p_j}\right) \cdot\frac {1} {p_k} \cdot \left(\prod_{j=k+1}^\infty \frac {p_j-q_j+1} {p_j}\right)=\sum_{k=1}^\infty \left( \prod_{j=1}^{k-1} \sum_{E_j=0}^{q_j-1} \frac {\xi_j(E_j)} {p_j} \right) \cdot \left( \sum_{E_k=1}^{q_k-1} \frac {1} {p_k} \cdot \Upsilon_k\right)\\
=&\sum_{k=1}^\infty \left(\prod_{j=1}^{k-1} \frac {p_j} {p_j} \right) \cdot \frac {q_k-1} {p_k}\cdot \Upsilon_k
=\sum_{k=1}^\infty \frac {q_k-1} {p_k} \cdot \prod_{j=k+1}^\infty \left(1- \frac {q_j-1} {p_j}\right).
\end{align*}

\end{proof}

\subsection{Measures on $\NNC$}
Let $\tau:\NN \to \NN$ be the left shift on $\NN$ and let $\measeNN$ (resp. $\measwNN$) be the collection of all ergodic (resp. weakly mixing) $\tau$-invariant Borel probability measures on $\NN$.  For $s, t \in \mathbb{N}$, set $\sigma_{s,t}=\tau^s \times \tau^t$ and $\sigma=\sigma_{1,1}$.  If $\omega=((p_1,p_2,\cdots),(q_1,q_2,\cdots)) \in \NNC$, then we write $\pione{\omega}=p_1$ and $\pitwo{\omega}=q_1$.  Similarly, if $\omega=(p_1,p_2,\cdots) \in \NN$, then we let $\pi(\omega)=p_1$.

For $S \subseteq \mathbb{N}$, we say that $\nu \in \measwNN$ is {\it positive on $S$} if $\nu\left( \left\{ \omega \in \NN : \pi(\omega) \in S \right\} \right) > 0$.  $\nu$ is {\it eventually positive} if there exists $M$ such that $\nu$ is positive on $\{n\}$ for all $n \geq M$.  If $\tau$ is weakly mixing, then $\sigma_{s,t}$ is ergodic and weakly mixing.
\begin{lem}\labl{firstae}
Suppose that $\mu_1, \mu_2 \in \measwNN$, $\mu=\mu_1 \times \mu_2$, and 
$$
\max\left(\ELone,\ELtwo\right)<\infty.
$$
If $\ELone>\al \cdot \ELtwo$ for $\al>1$, then for all integers $k \geq 0$ and $\mu$-almost every $\PQ \in \NNC$
$$
\lim_{n \to \infty} \frac {\p{n}} {\q{n\floor{\al}+k}}=\infty.
$$
If $\ELtwo>\al \cdot \ELone$ for $\al>1$, then for all integers $k \geq 0$ and $\mu$-almost every $\PQ \in \NNC$
$$
\lim_{n \to \infty} \frac {\p{n\floor{\al}+k}} {\q{n}}=0.
$$
\end{lem}
\begin{proof}
For integers $s,t \geq 1$, set
$$
f_{s,t,u,v}(\omega)=\sum_{j=0}^{t-1} \log \pitwo{\sigma^{j+v}(\omega)}-\sum_{j=0}^{s-1} \log \pione{\sigma^{j+u}(\omega)}.
$$
Let $t=\floor{\al}$ and note that
$$
\log \frac{\q{nt+k}} {\p{n}}=\log \q{k}+\sum_{j=1}^n \log \left(\frac {\prod_{m=1}^t q_{(j-1)t+k+m}} {p_j}\right).
$$
But $\sigma_{1,t}$ is ergodic, so for $\mu$-almost every $\omega \in \NNC$
\begin{align*}
&\lim_{n \to \infty} \frac {1} {n} \log \frac {\pitwo{\omega} \cdots \pitwo{\sigma^{nt+k-1}(\omega)}} {\pione{\omega}\cdots \pione{\sigma^{n-1}{\omega}}}
=\lim_{n \to \infty} \left(\frac {1} {n} \cdot \pitwo{\omega} \cdots \pitwo{\sigma^{k-1}(\omega)} +\frac {1} {n} \sum_{i=0}^{n-1} f_{1,t,0,k}\circ \sigma_{1,t}^i(\omega)\right)\\
=&0+\int f_{1,t,0,k}(\omega) \ d\mu(\omega)
=\int \left(\sum_{j=0}^{t-1} \log \pitwo{\sigma^{j+k}(\omega)}-\log \pione{\omega} \right) \ d\mu(\omega)\\
=&\left(\sum_{j=0}^{t-1} \int \log \pitwo{\sigma^{j+k}(\omega)} d\mu(\omega)\right)-\int \log \pione{\omega} \ d\mu(\omega)
=\left(\sum_{j=0}^{t-1} \int \log \pitwo{\omega} d\mu(\omega)\right)-\int \log \pione{\omega} \ d\mu(\omega)\\
=&t \ELtwo-\ELone<\al \ELtwo-\ELone<0.
\end{align*}
Thus, for $\mu$-almost every $\omega \in \NNC$
$$
\lim_{n \to \infty} \log \frac {\pitwo{\omega} \cdots \pitwo{\sigma^{nt+k-1}(\omega)}} {\pione{\omega}\cdots \pione{\sigma^{n-1}{\omega}}}=-\infty,
$$
so
$$
\lim_{n \to \infty} \frac {\pitwo{\omega} \cdots \pitwo{\sigma^{nt+k-1}(\omega)}} {\pione{\omega}\cdots \pione{\sigma^{n-1}({\omega})}}=0
$$
and the first assertion follows.  The second assertion is proven similarly.
\end{proof}

\begin{lem}
Suppose that $\max\left(\ELone,\ELtwo\right)<\infty$ and $\al \in [0,1]$.  If $\ELone<\al \ELtwo$, then
$$
\lim_{n \to \infty} \frac {(p_1 p_2 \cdots p_n)^\alpha}  {q_1 q_2 \cdots q_n} \cdot \min(p_n,q_n)^{1-\al} =0 
$$
and
$$
\lim_{n \to \infty} \frac {(p_1 p_2 \cdots p_{n+k})^\alpha}  {q_1 q_2 \cdots q_n} \cdot \left( \frac {p_{n+k+1}} {\max(1,p_{n+k+1}-q_{n+k+1})}\right)^\al =0,
$$
for $\mu$-almost every $\PQ \in \NNC$.
\end{lem}

\begin{proof}
The proof is similar to the proof of \refl{firstae} after we note that $\min(p_n,q_n)^{1-\al} \leq q_n^{1-\al} \leq q_n$ and $\left( \frac {p_{n+k+1}} {\max(1,p_{n+k+1}-q_{n+k+1})}\right)^\al \leq p_{n+k+1}^\al$.
\end{proof}

\begin{lem}\labl{liminferg}
If $\mu_1, \mu_2 \in \measwNN$, $\mu=\mu_1 \times \mu_2$, and $\ELone \leq \ELtwo<\infty$, then for $\mu$-almost every $\PQ \in \NNC$
$$
\liminf_{n \to \infty} \frac {\p{n}}{\q{n}}=0.
$$
\end{lem}

\begin{proof}
For $M>0$, let $A_M=\left\{\omega \in \NNC: \liminf_{n \to \infty} \frac {\pione{\omega}\cdots\pione{\sigma^{n-1}(\omega)}} {\pitwo{\omega}\cdots\pitwo{\sigma^{n-1}(\omega)}} \geq M \right\}$.
Assume for contradiction that $\mu(A_M)>0$.  Note that $A_M$ is a $\sigma$-invariant set, so $\mu(A_M)=1$ by the ergodicity of $\sigma$.  Let $f(\omega)=\pione{\omega}-\pitwo{\omega}$ and put $S_n(f)(\omega)=\sum_{i=0}^{n-1} f \circ \sigma^j(\omega)$.  Clearly, $\omega \in A_M$ if and only if $\liminf_{n \to \infty} S_n(f) \geq \log M$.  Thus, $\int \liminf_{n \to \infty} S_n(f)(\omega) d\mu(\omega)\geq \log M$.  Since 
$$
\int f(\omega) \ d\mu(\omega)=\ELone-\ELtwo \geq 0,
$$
we have that
\begin{align*}
&\liminf_{n \to \infty} \int S_n(f)(\omega) \ d\mu(\omega)
=\liminf_{n \to \infty} \int \left(\sum_{i=0}^{n-1} f \circ \sigma^i(\omega) \right) \ d\mu(\omega)
=\liminf_{n \to \infty} \sum_{i=0}^{n-1} \left( \int f \circ \sigma^i(\omega) \ d\mu(\omega) \right)\\
=&\liminf_{n \to \infty} \sum_{i=0}^{n-1} \left( \int \left( \pione{\sigma^{i}(\omega)}-\pitwo{\sigma^i(\omega)} \right) \ d\mu(\omega) \right)
=\liminf_{n \to \infty} \sum_{i=0}^{n-1} \left( \int \left(\pione{\omega}-\pitwo{\omega}\right) \ d\mu(\omega) \right)\\
= & \liminf_{n \to \infty} \sum_{i=0}^{n-1}  \int f(\omega) \ d\mu(\omega) 
\leq  \liminf_{n \to \infty} \sum_{i=0}^{n-1} 0 =0.
\end{align*}
By Fatou's lemma, $\int \liminf_{n \to \infty} S_n(f)(\omega) \ d\mu(\omega) \leq \liminf_{n \to \infty} \int S_n(f)(\omega) \ d\mu(\omega)$, which implies that $M \leq 0$, a contradiction.
\end{proof}

\begin{lem}\labl{BVergodic}
If $\max\left(\ELonesq,\ELtwosq\right)<\infty$, then
$$
\sum_{k=1}^\infty \sum_{j=k+1}^\infty \frac {p_k(p_j+q_j)} {q_1 \cdots q_j}<\infty
$$
for $\mu$-almost every $\PQ \in \NNC$.
\end{lem}
\begin{proof}
We will show that
\begin{equation}\labeq{intdoublesum1}
\int \sum_{k=1}^\infty \sum_{j=k+1}^\infty \frac {\pione{\sigma^{k-1}(\omega)}(\pione{\sigma^{j-1}(\omega)}+\pitwo{\sigma^{j-1}(\omega)})} {\pitwo{\omega} \cdots \pitwo{\sigma^{j-1}(\omega)}} \ d\mu(\omega)<\infty.
\end{equation}
Since each term in \refeq{intdoublesum1} is non-negative, the left hand side of \refeq{intdoublesum1} is equal to
\begin{align*}
&\sum_{k=0}^\infty \sum_{j=k}^\infty \int \frac {\pione{\sigma^{k}(\omega)}(\pione{\sigma^{j}(\omega)}+\pitwo{\sigma^{j}(\omega)})} {\pitwo{(\omega)} \cdots \pitwo{\sigma^{j}(\omega)}}d\mu(\omega) 
\leq \sum_{k=0}^\infty \sum_{j=k}^\infty \int \frac {\pione{\sigma^{k}(\omega)}(\pione{\sigma^{j}(\omega)}+\pitwo{\sigma^{j}(\omega)})} {2^j}d\mu(\omega)\\
\leq &\sum_{k=0}^\infty \sum_{j=k}^\infty 2^{-j} \left( \int \pione{\sigma^k(\omega)}^2 \ d\mu(\omega)\right)^{1/2}\left(\int (\pione{\sigma^j(\omega)}+\pitwo{\sigma^j(\omega)})^2 \ d\mu(\omega)\right)^{1/2} \hbox{ by Cauchy-Schwarz}\\
=& \sum_{k=0}^\infty \sum_{j=k}^\infty 2^{-j} \left( \int \pione{\omega}^2 \ d\mu(\omega)\right)^{1/2}\left(\int (\pione{\omega}+\pitwo{\omega})^2 \ d\mu(\omega)\right)^{1/2}\leq \sum_{k=0}^\infty \sum_{j=k}^\infty C \cdot 2^{-j}<\infty.
\end{align*}
\end{proof}

Lastly, we note the following trivial lemma.

\begin{lem}\labl{eventuallyposneq}
If $\mu_1, \mu_2 \in \measNN$ are eventually positive, then $p_n<q_n$ infinitely often and $p_n>q_n$ infinitely often for $\mu_1 \times \mu_2$-almost every $\PQ \in \NNC$.
\end{lem}


\subsection{Rationality of $\ppq$}

We will need the following theorem to discuss the rationality of $\ppq(x)$ for various $P, Q$, and $x$.  This theorem and a far more extensive discussion of the irrationality of sums of the form $\formalsum$ may be found in the monograph of J. Galambos \cite{Galambos} and is originally due to G. Cantor \cite{Cantor}.

\begin{thrm}\labt{rationaldivides}
Suppose that $Q$ has the property that for every positive integer $m$ there exist infinitely many positive integers $n$ such that $m | q_n$.  Then $\formalsum$ is rational if and only if $E_j=q_j-1$ for all but finitely many $j$ or if $E_j=0$, ultimately.\footnote{We remark that this sum isn't required to be a $Q$-Cantor series expansion.  That is, we may have $E_j=q_j-1$, ultimately.}
\end{thrm}

\begin{cor}\labc{rationaldivides}
Under the conditions of \reft{rationaldivides}, $\NU{Q}= \mathbb{Q}$ and $\U{Q}=\mathbb{R} \backslash \mathbb{Q}$.
\end{cor}

\begin{thrm}\labt{ppqrat}
Suppose that both $P$ and $Q$ have the property described in \reft{rationaldivides}.  Let $$S=\{x \in \mathbb{R} \backslash \mathbb{Q}  :   \ppq(x) \in \mathbb{Q}\}.$$
Then
\begin{enumerate}
\item $\ppq\left(\mathbb{Q} \right) \subseteq [0,1] \cap \mathbb{Q}$.
\item If  $p_n \leq q_n$ infinitely often, then $\ppq\left(\mathbb{R}\backslash \mathbb{Q} \right) \subseteq [0,1] \backslash \mathbb{Q}$.
\item If there exists $M=M(P,Q)$ such that $p_n \geq q_n$ for all $n \geq M$, but $p_n \geq q_n+1$ at most finitely often, then $S$ is countable and $S \cap [0,1)$ is finite.
\item If there exists $M=M(P,Q)$ such that $p_n \geq q_n$ for all $n \geq M$ and $p_n \geq q_n+1$ infinitely often, then $S \cap [0,1)$ is an uncountable dense set and 
$$
\lmeas{S \cap [0,1)}=\lim_{n \to \infty} \prod_{j=\max(M,n)}^\infty \frac {p_j-q_j} {p_j}  \in \{0,1\}.
$$
In particular, $\lmeas{S \cap [0,1)}=1$ if and only if $\sum \frac {q_j} {p_j}<\infty$.  Also,
\begin{equation}\labeq{dimhS}
 \liminf_{n \to \infty} \frac {\log \prod_{j=\max(M,n)}^n {(p_j-q_j)}} {\log \prod_{j=\max(M,n)}^n p_j} \leq \dimh{S}
\leq \liminf_{n \to \infty} \frac {\log \prod_{j=\max(M,n)}^n {(p_j-q_j+1)}} {\log \prod_{j=\max(M,n)}^n p_j}.
\end{equation}
If $\liminf_{n \to \infty} \frac {\log \prod_{j=\max(M,n)}^n {(p_j-q_j+1)}} {\log \prod_{j=\max(M,n)}^n p_j}<1$, then $\dimh{S}<\dimb{S}$.
\end{enumerate}
\end{thrm}
\begin{proof}
The first part follows directly from \refc{rationaldivides}.  Note that
$$
S=\{x=E_0.E_1E_2\cdots \hbox{ w.r.t. }P : \exists N \geq M \hbox{ such that } q_n-1 \leq E_n \leq p_n-1 \  \forall n>N \wedge E_n \neq p_n-1 \hbox{ infinitely often}\}.
$$
$S=\emptyset$ under the conditions of part (2).  Part (3) immediately follows from our characterization of $S$.  For part (4), we note that
\begin{equation}\labeq{rational}
\lim_{n \to \infty} \prod_{j=\max(M,n)}^\infty \frac {p_j-q_j} {p_j} \leq \lmeas{S \cap [0,1)}  \leq \lim_{n \to \infty} \prod_{j=\max(M,n)}^\infty \frac {p_j-q_j+1} {p_j}.
\end{equation}
The infinite products inside the limits in \refeq{rational} converge if and only if $\sum \frac {q_n} {p_n}$ and $\sum \frac {q_n+1} {p_n}$ converge, respectively.  Since $\frac {2} {3} \leq \frac {q_j} {q_j+1} <1$,
$$
\sum_{j=\max(M,n)}^\infty \frac {q_j} {p_j} < \sum_{j=\max(M,n)}^\infty \frac {q_j+1} {p_j} \leq \frac {3} {2} \cdot \sum_{j=\max(M,n)}^\infty \frac {q_j} {p_j}
$$
and either both $\sum \frac {q_j} {p_j}$ and $\sum \frac {q_j+1} {p_j}$ converge or they both diverge.  Thus, if $\sum \frac {q_j}{p_j}$ converges, then $$\lim_{n \to \infty} \prod_{j=\max(M,n)}^\infty \frac {p_j-q_j} {p_j} =1$$ and $\lmeas{S \cap [0,1)}=1$.  Otherwise, $\lmeas{S}=0$ by similar reasoning.
The expression for the Hausdorff dimension of $S$ follows by our characterization of $S$ and \reft{morewegmann}.  The set $S$ is dense in $\mathbb{R}$, so $\dimb{S}=1$ and the last statement follows from the estimate in \refeq{dimhS}.

\end{proof}
\reft{ppqrat} is given as only one example of a result on the rationality of $\ppq(x)$.  We should note that there are examples of $(P,Q) \in \NNC$ and $x \in \mathbb{Q}$ where $\ppq(x) \in [0,1] \backslash \mathbb{Q}$.  Let $p_n=3$ and $q_n=n+1$ for all $n$.  Put $x=0.1111\cdots$ w.r.t. $P=1/3$.  Then $\ppq(x)=e-2 \in [0,1] \backslash \mathbb{Q}$.

\begin{thrm}
If $\mu_1, \mu_2 \in \measNN$ are eventually positive, then $\ppq\left(\mathbb{Q}\right)~\subseteq~[0,1]~\cap~\mathbb{Q}$ and $\ppq\left(\mathbb{R}\backslash \mathbb{Q} \right) \subseteq [0,1] \backslash \mathbb{Q}$ for $\mu_1 \times \mu_2$-almost every $\PQ \in \NNC$.
\end{thrm}
\begin{proof}
This follows immediately from \refl{eventuallyposneq} and \reft{ppqrat}.
\end{proof}

\subsection{Continuity of $\ppq$}
Let 
\begin{align*}
&\LC:=\left\{x \in \mathbb{R} : \ppq \hbox{ is left continuous at }x\right\}, \\
&\mathscr{C}_{P,Q}^{\hbox{R}}:=\left\{x \in \mathbb{R} : \ppq \hbox{ is right continuous at }x\right\},\\
&\LD:=\mathbb{R} \backslash \LC, \RD:=\mathbb{R} \backslash \RC,\\
&\mathscr{C}_{P,Q}:=\LC \cap \RC, \hbox{ and } \mathscr{D}_{P,Q}:=\LD \cup \RD.
\end{align*}

\begin{lem}\labl{contlem}
Suppose that $t$ is a positive integer and $x=E_0.E_1E_2 \cdots E_t$ w.r.t. $P$, where $E_t \neq 0$.  Then $x \in \LC$ if and only if
\begin{equation}\labeq{contlem}
\min(E_t,q_t-1)-\min(E_t-1,q_t-1) = \sum_{j=t+1}^\infty \frac {\min(p_j-1,q_j-1)} {q_{t+1}q_{t+2} \cdots q_j}.
\end{equation}
\end{lem}
\begin{proof}
We rewrite \refeq{contlem} as
\begin{equation}\labeq{contlem2}
\sum_{j=1}^{t-1} \frac {\min(E_j,q_j-1)} {\q{j}}+\frac {\min(E_t,q_t-1)} {\q{t}}
=\sum_{j=1}^{t-1} \frac {\min(E_j,q_j-1)} {\q{j}}+\frac {\min(E_t-1,q_t-1)} {\q{t}}+\sum_{j=t+1}^\infty \frac {\min(p_j-1,q_j-1)} {\q{j}}.
\end{equation}
Let 
\begin{displaymath}
y_s=\left\{ \begin{array}{lll}
(E_0-1).(p_1-1)(p_2-1)\cdots (p_s-1) & \hbox{ w.r.t. } P & \textrm{if $x \in \mathbb{Z}$}\\
E_0.E_1E_2 \cdots E_{t-1} (E_t-1) (p_{t+1}-1) (p_{t+2}-1) \cdots (p_s-1)  &\hbox{ w.r.t. } P		& \textrm{if $x \notin \mathbb{Z}$}
\end{array} \right. 
\end{displaymath}
for $s > t$.
Clearly, $\lim_{s \to \infty} y_s=x$ and $y_s<x$.  We can rewrite \refeq{contlem2} as
\begin{equation}\labeq{contlem3}
\ppq(x) = \lim_{s \to \infty} \ppq(y_s).
\end{equation}
Since $y_s \to x$, $\ppq$ is not left continuous at $x$ if \refeq{contlem3} does not hold.  Now, suppose that \refeq{contlem3}  holds and let $(z_r)$ be any sequence of real numbers in $\mathbb{R}$ such that $z_r < x$ for all $r$ and $\lim_{r \to \infty} z_r=x$.  Then there exists a function $f(r)$ such that for large enough $r$, we have 
\begin{displaymath}
z_r=\left\{ \begin{array}{lll}
(E_0-1).(p_1-1)(p_2-1)\cdots (p_{f(r)-1}-1)F_{f(r)+1}F_{f(r)+2}\cdots & \hbox{ w.r.t. } P & \textrm{if $x \in \mathbb{Z}$}\\
E_0.E_1E_2 \cdots E_{t-1} (E_t-1) (p_{t+1}-1) (p_{t+2}-1) \cdots (p_{f(r)}-1)F_{f(r)+1}F_{f(r)+2}\cdots  &\hbox{ w.r.t. } P		& \textrm{if $x \notin \mathbb{Z}$}
\end{array} \right. .
\end{displaymath}
Then $|\ppq(z_r)-\ppq(y_{f(r)})| \to 0$, so $\ppq(z_r) \to \ppq(x)$ by \refeq{contlem3}.  Thus, $\ppq$ is left continuous at $x$. 
\end{proof}

For a positive integer $t$ and basic sequences $P$ and $Q$, let
\begin{align*}
&\mathscr{A}_{P,Q,t}:=\{E_0.E_1E_2\cdots E_t \hbox{ w.r.t. }P : E_t \geq q_t \};\\
&\mathscr{B}_{P,Q,t}:=\{E_0.E_1E_2\cdots E_{t-1} \hbox{ w.r.t. }P \};\\
&A_{P,Q}:=\{n : p_n>q_n\};\\
&B_{P,Q}:=\{n : p_n<q_n\}.
\end{align*}

\begin{thrm}\labt{ppqcont} $\RD=\emptyset$ and
$$
\LD=\left( \bigcup_{n \in A_{P,Q}} \mathscr{A}_{P,Q,n} \right) \cup \left( \bigcup_{n \in B_{P,Q}} \mathscr{B}_{P,Q,n} \right) \subseteq \mathscr{NU}_P.
$$
Moreover, $\ppq$ is lower semi-continuous on $\mathbb{R}$ if and only if $p_n \leq q_n$ whenever $n \geq 2$.  $\ppq$ is upper semi-continuous on $\mathbb{R}$ if and only if it is continuous on $\mathbb{R}$.
\end{thrm}
\begin{proof}
It is not difficult to see that $\ppq$ is continuous at all points in $\U{P}$ and right continuous on $\mathbb{R}$.
Let $x=E_0.E_1E_2\cdots E_t$ w.r.t. $P$ so
\begin{displaymath}
\min(E_t,q_t-1)-\min(E_t-1,q_t-1) =\left\{ \begin{array}{ll}
0		& \hbox{ if }E_t \geq q_t\\
1		& \hbox{ if }E_t < q_t
\end{array} \right. .
\end{displaymath}
Note that $\sum_{j=t+1}^\infty \frac {\min(p_j-1,q_j-1)} {q_{t+1}q_{t+2} \cdots q_j}>0$.  If $E_t\geq q_t$, then $x \in \LD$ by \refl{contlem}. This can only happen if $p_t>q_t$.  In case $E_t<q_t$, we see that $x \in \LD$ if and only if there exists some integer $s>t$ such that $p_s<q_s$ so that $\sum_{j=t+1}^\infty \frac {\min(p_j-1,q_j-1)} {q_{t+1}q_{t+2} \cdots q_j}<1$.  The semi-continuity can be analyzed with a slightly more careful argument that considers whether the jump discontinuities are positive or negative.
\end{proof}

\begin{cor}
The following are immediate consequences of \reft{ppqcont}.
\begin{enumerate}
\item $\mathscr{D}_{P,Q}$ is empty if and only if $p_1 \leq q_1$ and $p_t=q_t$ for all $t \geq 2$.  In this case, $\ppq(x)=\frac {p_1} {q_1} \cdot x$.
\item $\mathscr{D}_{P,Q}$ is at most finite if and only if $p_t \neq q_t$ at most finitely often.  Otherwise, $\mathscr{D}_{P,Q}$ is a countable dense subset of $\mathbb{R}$.
\item $\mathscr{D}_{P,Q}=\mathscr{NU}_P$ if and only if $B_{P,Q}$ is infinite.  Moreover, $\mathscr{D}_{P,Q}=\mathbb{Q}$ if $P$ satisfies the hypotheses of \reft{rationaldivides}.
\end{enumerate}
\end{cor}

\begin{thrm}
Suppose that $\mu_1, \mu_2 \in \measNN$ are eventually positive.  Then $\mathscr{D}_{P,Q}=\NU{P}= \mathbb{Q}$ for $\mu_1 \times \mu_2$-almost every $\PQ \in \NNC$.
\end{thrm}

\begin{thrm}\labt{ppqlinear}
Suppose that $p_n=q_n$ for all $n > t$.  Then $\ppq$ is piecewise linear.  In particular, for all $x=E_0.E_1E_2\cdots$ w.r.t. $P$
$$
\ppq(x)=\ppq(\{x\})
=\frac {p_1\cdots p_{t}} {q_1 \cdots q_{t}} \cdot \{x\}+\left(\sum_{n=1}^{t} \frac {\min(E_n,q_n-1)} {q_1\cdots q_n}-\frac {p_1 \cdots p_{t}} {q_1 \cdots q_{t}} \cdot \sum_{n=1}^{t} \frac {E_n} {p_1 \cdots p_n}   \right).
$$
\end{thrm}
\begin{proof}
Let $\al=\sum_{n=1}^{t} \frac {E_n} {p_1 \cdots p_n}$, $\be=\sum_{n=1}^{t} \frac {\min(E_n,q_n-1)} {q_1\cdots q_n}$, and $\ga=\sum_{n=t}^\infty \frac {E_n} {p_1\cdots p_n}$.  Since $\min(E_n,q_n-1)=E_n$ for $n > t$, we see that
$\ppq(\al+\ga)=\be+\frac {p_1 \cdots p_{t}} {q_1 \cdots q_{t}} \cdot \gamma$.  Thus,
\begin{align*}
&\be+\frac {p_1 \cdots p_{t}} {q_1 \cdots q_{t}} \cdot \gamma=\frac {p_1 \cdots p_{t}} {q_1 \cdots q_{t}} \cdot \al- \frac {p_1 \cdots p_{t}} {q_1 \cdots q_{t}} \cdot \al+\be+\frac {p_1 \cdots p_{t}} {q_1 \cdots q_{t}} \cdot \ga\\
=&\frac {p_1 \cdots p_{t}} {q_1 \cdots q_{t}} \cdot (\al+\ga)+\be-\frac {p_1 \cdots p_{t}} {q_1 \cdots q_{t}} \cdot \al\\
=&\frac {p_1 \cdots p_{t}} {q_1 \cdots q_{t}} \cdot \{x\}+\left(\sum_{n=1}^{t} \frac {\min(E_n,q_n-1)} {q_1\cdots q_n}-\frac {p_1 \cdots p_{t}} {q_1 \cdots q_{t}} \cdot \sum_{n=1}^{t} \frac {E_n} {p_1 \cdots p_n}   \right),
\end{align*}
and the conclusion follows.
\end{proof}

\subsection{Monotonicity, Bounded Variation, and Approximation of $\ppq$}\labs{monotone}
\begin{thrm}\labt{monotone}
$\ppq$ is monotone on no intervals if and only if $p_n>q_n$ infinitely often.
\end{thrm}
\begin{proof}
For simplicity, we only consider intervals contained in $[0,1)$
Suppose that $p_n>q_n$ infinitely often and let $J=[a,b] \subseteq [0,1)$ be a closed interval.  Then there exists an interval $I=[c,d] \subseteq J$ and $n \geq 1$ where $c=0.E_1E_2\cdots E_{n-1} (p_n-1)$ w.r.t. $P$ and $d=c+\frac {1} {p_1 \cdots p_n}$.  Let $m>n$ be such that $p_m>q_m$.  Set 
\begin{align*}
&x=0.E_1E_2\cdots E_{n-1} (p_n-1) \ 0 \ 0 \ 0 \cdots 0 \ (q_m-1) \ 1 &\hbox{ w.r.t. } P;\\
&y=0.E_1E_2\cdots E_{n-1} (p_n-1) \ 0 \ 0 \ 0 \cdots 0 \ q_m &\hbox{ w.r.t. } P.
\end{align*}
Clearly, $x,y \in I$, $c<x$, and $\ppq(c)<\ppq(x)$.  Also, $x < y$, but 
$$
\ppq(x)=\ppq(c)+\frac {q_m-1} {q_1 \cdots q_m}+\frac {1} {q_1 \cdots q_{m+1}}>\ppq(c)+\frac {q_m-1} {q_1 \cdots q_m}=\ppq(y).
$$
So, $\ppq$ is not monotone on the interval $J$.

Now, suppose that $p_n>q_n$ at most finitely often.  Let $M$ be large enough that $p_m \leq q_m$ for all $m \geq M$.  Consider the interval $I=\left[\sum_{n=1}^M \frac {p_n-1} {p_1 \cdots p_n},\sum_{n=1}^M \frac {p_n-1} {p_1 \cdots p_n}+\frac {1} {p_1 \cdots p_{M+1}} \right]$. It is easy to verify that $\ppq$ is increasing on this interval by applying \reft{ppqlinear}.
\end{proof}


\begin{cor}
Suppose that $p_n>q_n$ infinitely often and $|\lpq{w}| \notin \{0,1\}$.  Then $\lpq{w}$ is a totally disconnected set.
\end{cor}

\begin{thrm}
Suppose that $\mu_1, \mu_2 \in \measNN$ are eventually positive.  Then $\ppq$ is monotone on no intervals for $\mu_1 \times \mu_2$-almost every $\PQ \in \NNC$.
\end{thrm}

Given basic sequences $P$ and $Q$, let $P_t=(p_1,p_2,\cdots,p_t,2,2,2,\cdots)$ and $Q_t=(q_1,q_2,\cdots,q_t,2,2,2,\cdots)$.  

\begin{thrm}\labt{unifconvppqt}
The sequence of functions $\left(\ppqt\right)$ converges uniformly to $\ppq$ on $\mathbb{R}$.
\footnote{Only pointwise convergence of $\left(\ppqt\right)$ to $\ppq$ is used in this paper.}
\end{thrm}
\begin{proof}
Let $x=E_0.E_1E_2 \cdots$ w.r.t. $P$.  By \reft{ppqlinear}, 
\begin{align*}
\ppqt(x)&=\frac {p_1\cdots p_{t}} {q_1 \cdots q_{t}} \cdot \{x\}+\left(\sum_{n=1}^{t} \frac {\min(E_n,q_n-1)} {q_1\cdots q_n}-\frac {p_1 \cdots p_{t}} {q_1 \cdots q_{t}} \cdot \sum_{n=1}^{t} \frac {E_n} {p_1 \cdots p_n}   \right)\\
&=\sum_{n=1}^{t} \frac {\min(E_n,q_n-1)}{q_1 \cdots q_n}+\frac {p_1 \cdots p_{t}} {q_1 \cdots q_{t}} \sum_{n=t+1}^\infty \frac {E_n} {p_1 \cdots p_n}.
\end{align*}
Thus,
$$
\left| \ppq(x) - \ppqt(x)\right|
=\sum_{n=t+1}^{\infty} \frac {\min(E_n,q_n-1)}{q_1 \cdots q_n}+\frac {p_1 \cdots p_{t}} {q_1 \cdots q_{t}} \sum_{n=t+1}^\infty \frac {E_n} {p_1 \cdots p_n}
\leq \frac {1} {q_1 \cdots q_{t}}+\frac {1} {q_1 \cdots q_{t}} \leq \frac {1} {2^{t-1}}
$$
and $\left(\ppqt\right)$ converges uniformly to $\ppq$.
\end{proof}

\begin{cor}
\begin{align*}
&\int_0^1 \ppqt(x) \ dx=\frac {1} {2q_1 \cdots q_{t}}+\sum_{E_1=0}^{p_1-1} \sum_{E_2=0}^{p_2-1} \cdots \sum_{E_{t}=0}^{p_{t}-1} \left(\frac {1} {p_1 \cdots p_{t}} \sum_{n=1}^{t} \frac{\min(E_n,q_n-1)}{q_1\cdots q_n}-\frac {1} {q_1 \cdots q_{t}} \sum_{n=1}^{t} \frac {E_n} {p_1 \cdots p_n}     \right);\\
&\int_0^1 \ppq(x) \ dx=\lim_{t \to \infty} \sum_{E_1=0}^{p_1-1} \sum_{E_2=0}^{p_2-1} \cdots \sum_{E_{t}=0}^{p_{t}-1} \left(\frac {1} {p_1 \cdots p_{t}} \sum_{n=1}^{t} \frac{\min(E_n,q_n-1)}{q_1\cdots q_n}-\frac {1} {q_1 \cdots q_{t}} \sum_{n=1}^{t} \frac {E_n} {p_1 \cdots p_n}     \right).
\end{align*}
\end{cor}

\begin{proof}
The first assertion follows from computing the areas of the trapezoids bounded by pieces of the functions $\ppq$.  The latter assertion follows from the former, the dominated convergence theorem, and \reft{unifconvppqt}.
\end{proof}

We let $V(I,f)$ denote the total variation of the function $f$ on the closed interval $I$.  We say that $f$ is of {\it bounded variation} on $I$ if $V(I,f)<\infty$ and write $f \in \BV{I}$.  We will need the following well known theorem from \cite{Cesari}.

\begin{thrm}\labt{varliminf}
$V(I,\cdot):\BV{I} \to \mathbb{R}$ is a lower semi-continuous functional.  That is, if $(f_n)$ converges to $f$ pointwise on a closed interval $I$,  then
$$
V(I,f) \leq \liminf_{n \to \infty} V(I,f_n).
$$
\end{thrm}

Let $f(x^-):=\lim_{y \to x^-} f(y)$ denote the limit of $f(y)$ as $y$ approaches $x$ from the left.  We will also need the following lemma which is easily proven.

\begin{lem}\labl{rcontvar}
Suppose that $f:[a,b] \to \mathbb{R}$ is a piecewise monotone function that is right continuous on the non-empty closed interval $[a,b]$ with points of left discontinuity $x_1, x_2, \cdots, x_{r-1}$.  If $x_0=a$ and $x_r=b$,  then
$$
V([a,b],f)=\sum_{j=0}^{r-1} |f(x_j)-f(x_{j+1}^-)|+\sum_{j=1}^r |f(x_j)-f(x_j^-)|.
$$
\end{lem}

\begin{lem}\labl{ppqtvar}
If $t \geq 2$ and $p_t \neq q_t$, then
\begin{align*}
V\left([0,1],\ppqt\right)&=\sum_{k=1}^t \sum_{E=1}^{p_k-1} \left| \frac {\min(E,q_k-1)-\min(E-1,q_k-1)} {q_1 \cdots q_k}-\sum_{j=k+1}^t \frac {\min(p_j-1,q_j-1)} {q_1 \cdots q_j}- \frac {1} {q_1 \cdots q_t} \right|\\
&\ +\ppqt(1^-)+\frac {p_1 \cdots p_t} {q_1 \cdots q_t}
< 2 \cdot \sum_{k=1}^t \sum_{j=k+1}^t \frac {p_k (p_j+q_j)} {q_1 \cdots q_j}+2 \cdot \frac {p_1 \cdots p_t} {q_1 \cdots q_t}+1.
\end{align*}
\end{lem}

\begin{proof}
By \reft{ppqlinear}, $\ppqt$ is a piecewise linear function with slope $\frac {p_1 \cdots p_t} {q_1 \cdots q_t}$, which contributes $\frac {p_1 \cdots p_t} {q_1 \cdots q_t} \cdot (1-0)=\frac {p_1 \cdots p_t} {q_1 \cdots q_t}$ to the total variation of $\ppqt$.  Thus, by \refl{rcontvar}, we need only add this term to the sum of the magnitude of the jumps at the points of discontinuity of $\ppqt$.  
Since $p_t \neq q_t$, $\mathscr{D}_{P_t,Q_t}^{\hbox{L}} \subseteq \mathscr{B}_{P_t,Q_t,t+1}$ by \reft{ppqcont}.  If $x=1$, then $\ppqt(x)=0$, so $|\ppqt(x)-\ppqt(x^-)|=\ppqt(1^-)$.  
If $x=E_0.E_1E_2 \cdots E_k$ w.r.t. $P_t \in \mathscr{B}_{P_t,Q_t,t+1}$, where $E_k \neq 0$, then $\ppqt(x)=\sum_{n=1}^{k-1} \frac {\min(E_n,q_n-1)} {q_1 \cdots q_n}+\frac {\min(E_k,q_k-1)} {q_1 \cdots q_k}$ and
\begin{align*}
\ppqt(x^-)&=\sum_{n=1}^{k-1} \frac {\min(E_n,q_n-1)} {q_1 \cdots q_n}+\frac {\min(E_k-1,q_k-1)} {q_1 \cdots q_k}+\sum_{j=k+1}^t \frac {\min(p_j-1,q_j-1)} {q_1 \cdots q_j}+\sum_{j=t+1}^\infty \frac {1} {q_1 \cdots q_t \cdot 2^{j-t}}\\
& = \sum_{n=1}^{k-1} \frac {\min(E_n,q_n-1)} {q_1 \cdots q_n}+\frac {\min(E_k-1,q_k-1)} {q_1 \cdots q_k}+\sum_{j=k+1}^t \frac {\min(p_j-1,q_j-1)} {q_1 \cdots q_j}+\frac {1} {q_1 \cdots q_t}.
\end{align*}
Thus,
$$
\ppqt(x)-\ppqt(x^-)=\frac {\min(E_k,q_k-1)-\min(E_k-1,q_k-1)} {q_1 \cdots q_k}-\sum_{j=k+1}^t \frac {\min(p_j-1,q_j-1)} {q_1 \cdots q_j}- \frac {1} {q_1 \cdots q_t}.
$$
So, $\ppqt(x)-\ppqt(x^-)$ depends only on $k$ and the value of $E_k > 0$.  So we only need sum over values of $k$ and $E_k$ and the first part of the lemma follows.

To prove the inequality, we apply the triangle inequality to the term in the double summation.  
First, it is clear that $\min(E,q_k-1)-\min(E-1,q_k-1) \leq 1$, so
\begin{align*}
&\sum_{k=1}^t \sum_{E=1}^{p_k-1} \frac {\min(E,q_k-1)-\min(E-1,q_k-1)} {q_1 \cdots q_k} \leq \sum_{k=1}^t \sum_{E=1}^{p_k-1} \frac {1} {q_1 \cdots q_k} < \sum_{k=1}^t\frac {p_k} {q_1 \cdots q_t}\\
<& \sum_{k=1}^t \frac {p_k (p_{k+1}+q_{k+1})} {q_1 \cdots q_{k+1}} < \sum_{k=1}^t \sum_{j=k+1}^\infty \frac {p_k (p_j+q_j)} {q_1 \cdots q_j}.
\end{align*}
Next, $\min(p_j-1,q_j-1)<p_j+q_j$, so
$$
\sum_{k=1}^t \sum_{E=1}^{p_k-1} \sum_{j=k+1}^t \frac {\min(p_j-1,q_j-1)} {q_1 \cdots q_j} < \sum_{k=1}^t  \sum_{j=k+1}^t \frac {p_k(p_j+q_j)} {q_1 \cdots q_j}
$$
Lastly, $\ppqt(1^-) \leq 1$ and
$$
\sum_{k=1}^t \sum_{E=1}^{p_k-1} \frac {1} {q_1 \cdots q_t}<\frac {p_1 \cdots p_t} {q_1 \cdots q_t},
$$
so the second part of the lemma follows.
\end{proof}

\begin{thrm}\labt{BVppq}
If $I \subseteq \mathbb{R}$ is a non-empty closed interval, then
$\ppq \in \BV{I}$ if
\begin{align*}
\sum_{k=1}^\infty \sum_{j=k+1}^\infty \frac {p_k(p_j+q_j)} {q_1 \cdots q_j}<\infty \hbox{ and } \liminf_{t \to \infty} \frac {p_1 \cdots p_t} {q_1 \cdots q_t}< \infty.
\end{align*}
\end{thrm}
\begin{proof}
This follows immediately from \reft{unifconvppqt}, \reft{varliminf}, \refl{ppqtvar}, and the $1$-periodicity of $\ppq$.


\end{proof}

\begin{thrm}
Suppose that $I \subseteq \mathbb{R}$ is a closed interval, $\mu_1, \mu_2 \in \measwNN$ and $\mu=\muonetwo$.  If $\ELone \leq \ELtwo$, then $\ppq$ is of bounded variation for $\mu$-almost every $\PQ \in \NNC$.  
\end{thrm}

\begin{proof}
This follows from \refl{liminferg}, \refl{BVergodic}, and \reft{BVppq}.
\end{proof}

\subsection{Lipschitz and H\"older continuity of $\phpq{k}$}\labs{Holder}
We will need to analyze the H\"older continuity of $\phpq{k}$ in order to prove \reft{refinedWWX}.  $\zpq{k}$ will be non-empty as long as $\liminf_{n \to \infty} \min(p_n,q_n)\geq 3$.  Thus, we will require this assumption for every result in this subsection.

Note that 
\begin{equation}\labeq{zpqlist}
\zpq{0} \subsetneq \zpq{1} \subsetneq \zpq{2} \subsetneq \cdots \subsetneq \bigcup_{k=0}^\infty \zpq{k} \subsetneq \zpq{\infty}  \subsetneq [0,1) 
\end{equation}
and $\phpq{k}:\zpq{k} \to \zqp{k}$.

\begin{thrm}\labt{phpqhomeomorphism}
Suppose that $\liminf_{n \to \infty} \min(p_n,q_n)\geq 3$. Then $\phpq{k}$ is a homeomorphism from $\zpq{k}$ to $\zqp{k}$ for all $k \in \mathbb{N}_0$.
\end{thrm}
\begin{proof}
It is easy to see that $\phpq{k}$ is a bijection.  $\phpq{k}=\ppq|_{\zpq{k}}$ is continuous as $\zpq{k}$ may only be discontinuous on $\NU{P}$ by \reft{ppqcont} and $\NU{P} \cap \zpq{k} = \emptyset$ for all $k<\infty$.
Additionally, $\left(\phpq{k}\right)^{-1}$ is continuous as $\left(\phpq{k}\right)^{-1}=\phqp{k}$.
\end{proof}

\begin{lem}\labl{Holderestimate}
Suppose that $\liminf_{n \to \infty} \min(p_n,q_n)\geq 3$, $\al \in (0,1], k \in \mathbb{N}_0, x, y \in \zpq{k}$, and $x \neq y$.  Then for some constant $C(\al)$,
\begin{align*}
\frac {|\phpq{k}(x)-\phpq{k}(y)|} { |x-y|^\alpha} \leq C(\al) \sup_{n \in \mathbb{N}}  \max\Bigg(& \frac {(\p{n})^\alpha}  {\q{n}} \cdot \min(p_n,q_n)^{1-\al},\\
&\frac {(\p{n+k})^\alpha}  {\q{n}} \cdot \left( \frac {p_{n+k+1}} {\max(1,p_{n+k+1}-q_{n+k+1})}\right)^\al \Bigg).
\end{align*}
\end{lem}
\begin{proof}
Let $x=0.E_1E_2\cdots$ w.r.t. $P$, $y=0.F_1F_2\cdots$ w.r.t. $P$, $t=\min \{s : E_s \neq F_s\}$, and $G_n=E_n-F_n$.  Then
\begin{align*}
\frac {|\phpq{k}(x)-\phpq{k}(y)|} { |x-y|^\alpha} 
&=\frac {\left| \sum_{n=t}^\infty \frac {G_n} {\q{n}} \right|} {\left| \sum_{n=t}^\infty \frac {G_n} {\p{n}} \right|^\al}
\leq \frac { \sum_{n=t}^\infty \left|\frac {G_n} {\q{n}} \right|} {\left| \sum_{n=t}^\infty \frac {G_n} {\p{n}} \right|^\al}\\
& \leq \frac {\frac {|G_t|} {\q{t}}+\sum_{n=t+1}^\infty \frac {q_n-1} {\q{n}}} {\left(\frac {|G_t|} {\p{t}}-\left(\sum_{n=t+1}^{t+k} \frac {p_n-1} {\p{n}}\right)-\frac {\min(p_{t+k+1}-2,q_{t+k+1}-1)} {\p{t+k+1}}-\sum_{n=t+k+2}^\infty \frac {p_n-1} {\p{n}}\right)^\al}\\
&= \frac {\frac {|G_t|} {\q{t}}+\frac {1} {\q{t}}} {\left(\frac {|G_t|} {\p{t}}-\left( \frac {1} {\p{t}} - \frac {1} {\p{t+k}}\right) -\frac {\min(p_{t+k+1}-2,q_{t+k+1}-1)} {\p{t+k+1}}-\frac {1} {\p{t+k+1}}\right)^\al},
\end{align*}
which simplifies to
\begin{equation}\labeq{Holderest1}
\frac {\left(\frac {|G_t|+1} {\q{t}}\right)} {\left(\frac {|G_t|-1} {\p{t}}+\frac {\max(1,p_{t+k+1}-q_{t+k+1})} {\p{t+k+1}}\right)^\al}
\end{equation}
We now consider two cases.  First, if $|G_t|=1$, then \refeq{Holderest1} is equal to
\begin{equation}\labeq{Holderest2}
2\frac {(\p{t+k})^\al} {\q{t}}\cdot \left( \frac {p_{n+k+1}} {\max(1,p_{t+k+1}-q_{t+k+1})}\right)^\al.
\end{equation}
Let $C(\al)=\max\left(2,\sup_{w \geq 2} \frac {w+1} {w^{1-\al}(w-1)^\al}\right)$.  Clearly, $\frac {w+1} {w^{1-\al}(w-1)^\al}$ is  continuous for $w \geq 2$ and $\lim_{w \to \infty} \frac {w+1} {w^{1-\al}(w-1)^\al}=1$. Thus, $2 \leq C(\al)<\infty$.
Since $|G_t| < \min(p_t,q_t)$
\begin{equation}\labeq{Holderest4}
\frac {|G_t|+1} {(|G_t|-1)^\al}=\frac {|G_t|+1} {|G_t|^{1-\al}(|G_t|-1)^\al} \cdot |G_t|^{1-\al} \leq C(\al) |G_t|^{1-\al}<C(\al) \cdot \min(p_t,q_t)^{1-\al}.
\end{equation}
Suppose that $|G_t|>1$.  Using \refeq{Holderest4}, we may bound \refeq{Holderest1} above by 
\begin{equation}\labeq{Holderest3}
\frac {(\p{t})^\al} {\q{t}} \cdot \frac {|G_t|+1} {(|G_t|-1)^\al} 
\leq C(\al) \frac {(\p{t})^\al} {\q{t}} \cdot \min(p_t,q_t)^{1-\al},
\end{equation}
Combining the estimates \refeq{Holderest2} and \refeq{Holderest3} of \refeq{Holderest1}, the lemma follows.
\end{proof}

\begin{thrm}\labt{Holder}
Suppose that $k \in \mathbb{N}_0$ and $\liminf_{n \to \infty}\min(p_n,q_n) \geq 3$. Then $\phpq{k}$ is H\"older continuous of exponent $\al$ if
\begin{equation}\labeq{Holder1}
\limsup_{n \to \infty} \frac {(\p{n})^\alpha}  {\q{n}} \cdot \min(p_n,q_n)^{1-\al} < \infty 
\end{equation}
and
\begin{equation}\labeq{Holder2}
\limsup_{n \to \infty} \frac {(\p{n+k})^\alpha}  {\q{n}} \cdot \left( \frac {p_{n+k+1}} {\max(1,p_{n+k+1}-q_{n+k+1})}\right)^\al <\infty.
\end{equation}
Additionally, $\phpq{k}$ is not H\"older continuous of exponent $\al$ if \refeq{Holder1} does not hold.
\end{thrm}
\begin{proof}
The H\"older continuity of $\phpq{k}$ given \refeq{Holder1} and \refeq{Holder2} follows directly from \refl{Holderestimate}.  Suppose that \refeq{Holder1} does not hold.  Let the sequence $(n_t)$ be given such that
\begin{equation}\labeq{Holdernt}
\frac {(\p{n_t})^\alpha}  {\q{n_t}} \cdot \min(p_{n_t},q_{n_t})^{1-\alpha} > t.
\end{equation}
Let $x_t=\sum_{m=1}^{n_t} \frac {1} {\p{m}}+\sum_{m=1}^\infty \frac {1} {\p{n_t+2m}}$ and 
$$y_t=\sum_{m=1}^{n_t} \frac {1} {\p{m}}+\frac {\min(p_{n_t}-1,q_{n_t}-1)} {\p{n_t}}+\sum_{m=1}^\infty \frac {1} {\p{n_t+2m}},$$
 so $x_t,y_t \in \zpq{k}$. Then
\begin{align*}
&\frac {|\phpq{k}(x_t)-\phpq{k}(y_t)|} { |x_t-y_t|^\alpha}=\frac {(\p{n_t})^\alpha}  {\q{n_t}} \cdot \min(p_{n_t}-1,q_{n_t}-1)^{1-\alpha}\\
&=\frac {(\p{n_t})^\alpha}  {\q{n_t}} \cdot \min(p_{n_t},q_{n_t})^{1-\alpha} \cdot \frac {\min(p_{n_t}-1,q_{n_t}-1)^{1-\alpha}} {\min(p_{n_t},q_{n_t})^{1-\alpha}}\\
&> t \cdot \left(1-\frac {1} {\min(p_{n_t},q_{n_t})} \right)^{1-\al}
\geq t \cdot \left( \frac {2} {3} \right)^{1-\al}.
\end{align*}
Thus, $\lim_{t \to \infty} \frac {|\phpq{k}(x_t)-\phpq{k}(y_t)|} { |x_t-y_t|^\alpha} = \infty$ and $\phpq{k}$ is not H\"older continuous of exponent $\al$.  
\end{proof}

A nontrivial application of \reft{Holder} is given in \refl{phqpHolder}.

\begin{cor}
Suppose that $k \in \mathbb{N}_0$ and $\liminf_{n \to \infty} \min(p_n,q_n) \geq 3$. Then $\phpq{k}$ is Lipschitz if
$$
\limsup_{n \to \infty} \frac {\p{n+k}}  {\q{n}} \cdot \frac {p_{n+k+1}} {\max(1,p_{n+k+1}-q_{n+k+1})} < \infty.
$$
$\phpq{k}$ is not Lipschitz if
$$
\limsup_{n \to \infty} \frac {\p{n}} {\q{n}}=\infty.
$$
\end{cor}

\begin{thrm}
Suppose that $\mu_1, \mu_2 \in \measwNN$,  $\mu_1$ and $\mu_2$ are not positive on $\{2\}$.  Put $\mu=\muonetwo$, let $\al \in (0,1)$, and suppose that $\max\left(\ELone,\ELtwo\right)<\infty$.  If $\ELtwo > \al \ELone$, then $\phpq{k}$ is H\"older continuous of exponent $\al$ for all $k \geq 0$ for $\mu$-almost every $\PQ \in \NNC$.  If $\ELtwo > \ELone$, then $\phpq{k}$ is Lipschitz continuous for $\mu$-almost every $\PQ \in \NNC$.
\end{thrm}

\section{Normal numbers with respect to the Cantor series expansions}\labs{normal}
\subsection{Introduction}

Let
$$
Q_n^{(k)}:=\sum_{j=1}^n \frac {1} {q_j q_{j+1} \ldots q_{j+k-1}} \hbox{ and }  T_{Q,n}(x):=\left(\prod_{j=1}^n q_j\right) x \pmod{1}.
$$
A. R\'enyi \cite{Renyi} defined a real number $x$ to be {\it normal} with respect to $Q$ if for all blocks $B$ of length $1$,
\begin{equation}\labeq{rnormal}
\lim_{n \rightarrow \infty} \frac {N_n^Q (B,x)} {Q_n^{(1)}}=1.
\end{equation}
If $q_n=b$ for all $n$ and we restrict $B$ to consist of only digits less than $b$, then \refeq{rnormal} is equivalent to {\it simple normality in base $b$}, but not equivalent to {\it normality in base $b$}. A basic sequence $Q$ is {\it $k$-divergent} if
$
\lim_{n \rightarrow \infty} Q_n^{(k)}=\infty.
$
$Q$ is {\it fully divergent} if $Q$ is $k$-divergent for all $k$ and {\it $k$-convergent} if it is not $k$-divergent.  A basic sequence $Q$ is {\it infinite in limit} if $q_n \rightarrow \infty$.

\begin{definition}\labd{1.7} A real number $x$  is {\it $Q$-normal of order $k$} if for all blocks $B$ of length $k$,
$$
\lim_{n \rightarrow \infty} \frac {N_n^Q (B,x)} {Q_n^{(k)}}=1.
$$
We let $\Nk{Q}{k}$ be the set of numbers that are $Q$-normal of order $k$.  $x$ is {\it $Q$-normal} if
$
x \in \NQ := \bigcap_{k=1}^{\infty} \Nk{Q}{k}.
$
Additionally, $x$ is {\it simply $Q$-normal} if it is $Q$-normal of order $1$.  $x$ is {\it $Q$-ratio normal of order $k$} (here we write $x \in \RNk{Q}{k}$) if for all blocks $B_1$ and $B_2$ of length $k$
$$
\lim_{n \to \infty} \frac {N_n^Q (B_1,x)} {N_n^Q (B_2,x)}=1.
$$
$x$ is {\it $Q$-ratio normal} if
$
x \in \RNQ := \bigcap_{k=1}^{\infty} \RNk{Q}{k}.
$
A real number~$x$ is {\it $Q$-distribution normal} if
the sequence $(T_{Q,n}(x))_{n=0}^\infty$ is uniformly distributed mod $1$.  Let $\DNQ$ be the set of $Q$-distribution normal numbers.
\end{definition}

It is easy to show that $\DNQ$ is a set of full Lebesgue measure for every basic sequence $Q$. 
For $Q$ that are infinite in limit,
it has been shown that $\Nk{Q}{k}$ is of full measure if and only if $Q$ is $k$-divergent \cite{Mance4}.  Early work in this direction has been done by A. R\'enyi \cite{Renyi}, T.  \u{S}al\'at \cite{Salat4}, and F. Schweiger \cite{SchweigerCantor}.  Therefore if $Q$ is infinite in limit, then $\NQ$ is of full measure if and only if $Q$ is fully divergent.

\begin{wrapfigure}{R}{0.2\textwidth}
\begin{tikzpicture}[>=stealth',shorten >=1pt,node distance=2.5cm,on grid,initial/.style    ={}]
  \node[state]          (NQ)                        {$\mathsmaller{\NQ}$};
  \node[state]          (RNQ) [right =of NQ]    {$\mathsmaller{\RNQ}$};
  \node[state]          (DNQ) [above left=of RNQ]    {$\mathsmaller{\DNQ}$};
\tikzset{mystyle/.style={->,double=black}} 
\tikzset{every node/.style={fill=white}} 
\path       (NQ)     edge [mystyle]     (RNQ);
\end{tikzpicture}
\caption{}
\labf{figure2}
\end{wrapfigure}

Note that in base~$b$, where $q_n=b$ for all $n$,
 the corresponding notions of $Q$-normality, $Q$-ratio normality, and $Q$-distribution normality are equivalent. This equivalence
is fundamental in the study of normality in base $b$. It is surprising that this
equivalence breaks down in the more general context of $Q$-Cantor series for general $Q$.

It is usually most difficult to establish a lack of a containment relationship.  The first non-trivial result in this direction was in \cite{AlMa} where a basic sequence $Q$ and a real number $x$ is constructed where $x \in \NQ \backslash \DNQ$.\footnote{This real number $x$ satisfies a much stronger condition than not being $Q$-distribution normal: $T_{Q,n}(x) \to 0$.}  By far the most difficult of these to establish is the existance of a basic sequence $Q$ where $\RNQ \cap \DNQ \backslash \NQ \neq \emptyset$.  This case will be considered in the next subsection and requires information about the functions $\ppq$ established in the previous section.  \reft{NnotDN} provides a significant improvement over the main result of \cite{AlMa} while \reft{RNnotN} and \reft{DNnotRN} provide simpler proofs of known results using information about $\ppq$.  It was proven in \cite{Mance7} that $\dimh{\DNQ \backslash \RNk{Q}{1}}=1$ whenever $Q$ is infinite in limit.  It should be noted that most of the relations in \reff{figure1} are trivially induced by those in \reff{figure2}.

We note the following fundamental fact about $Q$-distribution normal numbers that follows directly from a theorem of T. \u{S}al\'at \cite{Salat}.\footnote{The original theorem of T. \u{S}al\'at says:
Given a basic sequence $Q$ and a real number $x$ with
$Q$-Cantor series expansion
$x=\floor{x}+\formalsum,$ if
$\lim_{N\to\infty} \frac {1} {N}\sum_{n=1}^N \frac {1} {q_n}=0$ then
$x$ is $Q$-distribution normal iff $E_n=\floor{\theta_n q_n}$ for some uniformly distributed sequence $(\theta_n)$.  N. Korobov \cite{Korobov} proved this theorem under the stronger condition that $Q$ is infinite in limit.  For this paper, we will only need to consider the case where $Q$ is infinite in limit.}
\begin{thrm}\labt{Salat}
Suppose that $Q=(q_n)$ is a basic sequence and $\lim_{N\to\infty} \frac {1} {N}\sum_{n=1}^N  \frac{1} {q_n}=0$. Then $x=E_0.E_1E_2\cdots$ w.r.t. $Q$ is $Q$-distribution normal if and only if $(E_n/q_n)$ is uniformly distributed mod 1.
\end{thrm}

The following immediate consequence of \reft{mainpsi} will be used in this section.

\begin{thrm}\labt{presrat}
Suppose that $Q_1,Q_2,\cdots,Q_j$ are infinite in limit and $\lim_{n \to \infty} N_n^{Q_1}((0),x)=\infty$.  Then
\begin{align*}
\Psi_j(\RNk{Q_j}{k}) \subseteq \RNk{Q_j}{k} \hbox{ and }
\Psi_j(\RNQ) \subseteq \RNQ.
\end{align*}
\end{thrm}
It should be noted that $\ppq$ does not preserve normality or distribution normality.  We will exploit this fact to construct a basic sequence $Q$ and a member of $\RDN$.  We will start with a basic sequence $P$ and a real number $\eta$ that is $P$-normal.  A basic sequence $Q$ will be carefully chosen so that $\ppq(\eta) \in \DNQ$, but $\ppq(\eta) \notin \NQ$.  Thus, we will be ``trading'' $P$-normality for $Q$-distribution normality.  \reft{presrat} will guarantee that $\ppq(\eta) \in \RNQ$.

We should note that not all constructions in the literature of normal numbers are of computable real numbers. For example, the construction by M. W. Sierpinski in \cite{Sierpinski} is not of a computable real number.  V. Becher and S. Figueira modified M. W. Sierpinski's work to give an example of a computable absolutely normal number in \cite{BecherFigueira}.  
 Since not every basic sequence is computable we face an added difficulty.  Moreover, many of the numbers we construct by using \reft{mainpsi} are not computable.  Thus, we will indicate when a number we construct is computable.  

\subsection{Explicit construction of a  basic sequence $Q$ and a member of $\RNQ \cap \DNQ \backslash \NQ$.}\labs{RDN}
\subsubsection{Some results on construction of distribution normal numbers}
Given blocks $B$ and $Y$, we let $N(B,Y)$ be the number of
occurrences of the block~$B$ in the block~$Y$.
Given a Borel probability measure $\mu$ on $\Nc{0}^{\mathbb{N}}$ and $B=(b_1,\cdots,b_k) \in \Nc{0}^k$, we write
$$
[B]=\left\{\omega=(\omega_1,\omega_2,\cdots) \in \Nc{0}^{\mathbb{N}} : \omega_j=b_j \forall j \in [1,k]       \right\} \hbox{ and }
\mu(B)=\mu\left([B]\right).
$$
A block of digits $Y$ is~{\it $(\epsilon,k,\mu)$-normal} if for all blocks $B$ of length $m \leq k$, we have 
$(1-\epsilon)|Y|\mu(B) \leq N(B,Y) \leq (1+\epsilon) |Y| \mu(B).$
Let $\lambda_b$ be any Borel probability measure on  $\Nc{0}^{\mathbb{N}}$ where $\lambda_b(B)=b^{-k}$ for all blocks $B$ of length $k$ in base $b$.
A {\it modular friendly family(MFF)}, $V$, is a sequence of triples
$((l_i,b_i,\e_i))_{i=1}^\infty$ such that
$(l_i)_{i=1}^\infty$ and $(b_i)_{i=1}^\infty$ are
non-decreasing sequences of non-negative integers with $b_i\ge 2$,
such that $(\e_i)_{i=1}^\infty$ is a decreasing sequence of real
numbers in $(0,1)$ with $\lim_{i\to\infty}\e_i=0$.
A sequence $(X_i)_{i=1}^\infty$ of $(\e_i,1,\lambda_{b_i})$-normal blocks of non-decreasing length
with $\lim_{i\to\infty}|X_i|=\infty$ is
 {\it $V$-nice} if $\frac {l_{i-1}} {l_i}\cdot \frac {|X_{i-1}|} {|X_i|}=o(1/i)$ and $\frac {1} {l_i}\cdot \frac {|X_{i+1}|} {|X_i|}=o(1)$.
Set $L_i=\left| X_1^{l_1} \ldots X_i^{l_i}\right|=l_1 |X_1|+\ldots+l_i |X_i|$, $s_n=b_i$ for $L_{i-1} < n \leq L_i$, 
$\Gamma(V,X):=(s_n)_{n=1}^{\infty}$, and $\eta(V,X):=\sum_{n=1}^{\infty} \frac {E_n} {s_1 \cdots s_n}$, where $(E_1,E_2,\ldots)=X_1^{l_1} X_2^{l_2} X_3^{l_3} \cdots$.

\begin{thrm}\labt{mqd}
Let $V=((l_i,b_i,\e_i))_{i=1}^\infty$ be an $MFF$
and suppose that 
$X=(X_i)_{i=1}^\infty$ is $V$-nice.
Then~$\eta(V,X)$ is $\Gamma(V,X)$-distribution normal.
\footnote{
Our statement of \reft{mqd} and the preceding definitions has been altered to be more concisely stated than they were in \cite{AlMa}.  We also removed some unnecessary hypotheses.
It was not stated in \cite{AlMa}, but it is not difficult to show that the conclusion of \reft{mqd} may be strengthened to say that $x(V,X) \in \DNQ \cap \Nk{Q}{1}$ by using the main theorem in \cite{Mance}.}
\end{thrm}

We will modify the construction of a basic sequence $P$ and a real number $x \in \N{P} \backslash \DN{P}$ given by C. Altomare and the author in \cite{AlMa}.
Let~$b$ be a positive integer. We define $\nu_b \in \measNNz$ as follows.  Put
\begin{displaymath}
\nu_b((j))=\left\{ \begin{array}{ll}
\frac {1} {2^b} & \textrm{if $0 \leq j \leq b-1$}\\
\frac {2^b-b} {2^b} & \textrm{if $j=b$}\\
0		& \textrm{if $j>b$}
\end{array} \right.
\end{displaymath}
and for a block $B=(b_1,\ldots,b_k)$, put
$\nu_b(B)=\prod_{j=1}^k \nu_b((b_j)).$
Let~$b$ and~$w$ be positive integers.  Let $V_1,V_2,\ldots,V_{(b+1)^w}$ be the blocks in base $b+1$
of length~$w$ written in lexicographic order. Put
$$
V_{b,w}=
V_1^{2^{bw}\nu_b(V_1)} V_2^{2^{bw}\nu_b(V_2)}\cdots V_{(b+1)^w}^{2^{bw}\nu_b(V_{(b+1)^w})}.
$$
With these definitions, we may state the following results from 
\cite{AlMa}.

\begin{thrm}\labt{qnex}
For $i\le 5$, let $X_i=(0,1)$, $b_i=2$, and $l_i=0$.
For $i\ge 6$, let $X_i=V_{i,i^2}$, $b_i=2^i$, and
$l_i=2^{4i^2}$. If  $V=((l_i,b_i,\e_i))_{i=1}^\infty$ and $X=(X_i)_{i=1}^\infty$, then~$\eta(V,X) \in \N{\Gamma(V,X)} \backslash \DN{\Gamma(V,X)}$.  Moreover, $\lim_{n\to\infty} T_{\Gamma(V,X),n}(\eta(V,X))=0$.
\end{thrm}

\subsubsection{The Construction}

We need the following lemma from \cite{AlMa}.
\begin{lem}
If $b$ and $w$ are positive integers, then $|V_{b,w}|=w2^{bw}$.
\end{lem}

Let $W_i=V_{i,i^2}=(w_{i,t})$, so $|W_i|=i^2 2^{i^3}$ when $i \geq 2$.
We first need to define sequences $Q_i=(q_{i,t})_{t=1}^{|W_i|}$ as follows.
Let $1 \leq t \leq |W_i|$.  If $w_{i,t} \in \{0,1,\cdots,i-1\}$, set $q_{i,t}=i$.  Thus there are 
$$
i^22^{i^3}-\frac {i} {2^i} \cdot i^22^{i^3}=i^22^{i^3}-i^32^{i^3-i}
$$
remaining values $q_{i,t}$ to assign where $w_{i,t}=i$.  Since $i | (i^22^{i^3}-i^32^{i^3-i})$, we may portion these into $i$ classes of $i2^{i^3}-i^22^{i^3-i}$ elements.  
In the first of these, we let $q_{i,t}=i^3$.  If $2 \leq j \leq i$, then we set $q_{i,t}=\floor{i^2/(j-1)}$ if $q_{i,t}$ is in the $j$'th grouping.  Set $Y_i=(y_{i,t})_{t=1}^{|W_i|}$, where
\begin{displaymath}
y_{i,t}=\left\{ \begin{array}{ll}
0 		& \hbox{if }w_{i,t}=0 \hbox{ or } (w_{i,t}=i \hbox{ and } q_{i,t}=i^3)\\
\al		& \hbox{if }w_{i,t}=\al \hbox{ or } (w_{i,t}=i \hbox{ and } q_{i,t}=\floor{i^2/\al})
\end{array} \right. .
\end{displaymath}
We note the following lemma which follows immediately from construction.
\begin{lem}\labl{RDNcount}
\begin{align*}
&N((t),W_i)=\left\{ \begin{array}{ll}
\frac {1} {2^i} |W_i|		& \hbox{if }0 \leq t \leq i-1\\
\frac {2^i-i} {2^i} |W_i|	& \hbox{if }t=i\\
0				& \hbox{if }t>i
\end{array} \right.
=\left\{ \begin{array}{ll}
i^2 2^{i^3-i}			& \hbox{if }0 \leq t \leq i-1\\
i2^{i^3}			& \hbox{if }t=i\\
0				& \hbox{if }t>i
\end{array} \right.;\\
& |\{n: E_{i,n}=t<i\}|=i^2 2^{i^3-i};\\
& |\{n: E_{i,n}=i \hbox{ and } q_{i,n}=i^3\}|=i2^{i^3}-i^22^{i^3-i};\\
& |\{n: E_{i,n}=i \hbox{ and } q_{i,n}=\floor{i^2/\al}\}|=i2^{i^3}-i^22^{i^3-i}.
\end{align*}
\end{lem}

\begin{lem}\labl{Yi}
$|Y_i|=|W_i|=i^22^{i^3}$ and $Y_i$ is $(0,1,\lambda_i)$-normal.
\end{lem}
\begin{proof}
$|Y_i|=|W_i|$ follows immediately by construction.  Let $j \in \{0,\ldots,i-1\}$.  By \refl{RDNcount}, $N(Y_i,(j))=i^2 2^{i^3-i}+(i2^{i^3}-i^22^{i^3-i})=i2^{i^3}=\frac {1} {i}|Y_i|$.
\end{proof}

\begin{lem}\labl{othernumdist}
For $i \leq 5$, let $X_i=(0,1)$, $b_i=2$, and $l_i=0$.  For $i \geq 6$, let $X_i=Y_i$, $b_i=i$, and $l_i=2^{4i^2}$.  Put $V=((l_i,b_i,\e_i))_{i=1}^\infty$ and $X=(X_i)_{i=1}^\infty$.  Then $\eta(V,X)$ is $\Gamma(V,X)$-distribution normal.
\end{lem}
\begin{proof}
This follows immediately from \reft{mqd} and \refl{Yi}.
\end{proof}

For the remainder of \refs{RDN}, we will define $P$ to be the basic sequence constructed in \reft{qnex} and refer to the number constructed in the same theorem as $\zeta$.\footnote{This number $\zeta$ has many pathological properties and is a reasonable starting place for constructing counterexamples.  A well known property of normal numbers in base $b$ is that $x$ is normal in base $b$ if and only if $rx$ is normal in base $b$ for all rational numbers $r$.  It is not difficult to see that $P$-normality is not even preserved by integer multiplication.  That is  $\zeta$ has the property that $n\zeta$ is not $P$-normal for every integer $n \geq 2$.}  We also refer to the number constructed in \refl{othernumdist} as $\kappa$ and the basic sequence as $K=(k_n)$. We will write $\kappa=0.F_1F_2\cdots$ w.r.t. $K$.
 Clearly, the sequence $(\al_n)=(T_{K,n}(\kappa))$ is uniformly distributed mod $1$ since $\kappa \in \DN{K}$.  We will construct a basic sequence $Q$  such that $(\beta_n)=\left(T_{Q,n}\left(\phpq{k}(\zeta)\right)\right)$ has the property that $\be_n-\al_n \to 0$.  This will establish that $\ppq(\zeta)$ is in $\DNQ$.  Additionally, we will show that for our choice of $Q$, we will have $\ppq(\zeta) \in \RN{Q} \backslash \N{Q}$.  Let $\Delta_{t,i}=\left[\frac {t} {i},\frac {t+1} {i}    \right)$.
\begin{lem}\labl{inDelta}
If $0<\al \leq i-1$, then
$\frac {i} {\floor{i^2/\al}} \in \Delta_{\al,i}$.
\end{lem}
\begin{proof}
First, we note that $\frac {i} {\floor{i^2/\al}} \geq \frac {i} {i^2/\al}=\frac {\al}{i}$.  
In order to finish the proof, we need to show that 
\begin{equation}\labeq{DeltaGoal}
\frac {i} {\floor{i^2/\al}} < \frac{\al+1} {i}.
\end{equation}
We see that
$$
\frac {i} {\floor{i^2/\al}} \leq \frac {i} {\frac {i^2}{\al}-1}=\frac {\al i} {i^2-\al},
$$
so
$$
\frac {\al+1} {i}-\frac {i} {\floor{i^2/\al}} \geq \frac {\al+1} {i}-\frac {\al i} {i^2-\al}=\frac {i^2-\al^2-\al} {i(i^2-\al)}
\geq \frac {i^2-(i-1)^2-(i-1)} {i(i^2-0)}=\frac {1} {i^2}>0,
$$
establishing \refeq{DeltaGoal}.
\end{proof}


\begin{thrm}\labt{RDNnonempty}
Put $Q=Q_6^{l_6}Q_7^{l_7}Q_8^{l_8}\cdots$, where $l_i=2^{4i^2}$.  Then $Q$ is infinite in limit and fully divergent and $\ppq(\zeta) \in \RDN$.
\end{thrm}
\begin{proof}
For $n \in \mathbb{N}$, let $i=i(n)$ be the unique integer such that $l_6|X_6|+\cdots+l_{i-1}|X_{i-1}| < n \leq l_6|X_6|+\cdots+l_i|X_i|$.
Note that by \refl{inDelta}, $\frac {E_n} {q_n} \in \Delta_{\al,i(n)}$ if and only if $\frac {F_n} {k_n} \in \Delta_{\al,i(n)}$.  Thus,
$$
\left| T_{Q,n}\left(\ppq(\zeta)\right) - T_{K,n}(\kappa)\right| 
< \left|\frac {E_{n+1}} {q_{n+1}}-\frac{F_{n+1}}{k_{n+1}}\right|+\frac {1} {q_{n+1}}+\frac {1} {k_{n+1}}
\leq \frac {1} {i(n+1)}+\frac {1} {q_{n+1}}+\frac {1} {k_{n+1}} \to 0.
$$
Since the sequence $\left(T_{K,n}(\kappa)\right)$ is uniformly distributed mod 1, we may conclude 
that $\left(T_{Q,n}\left(\ppq(\zeta)\right)\right)$ is uniformly distributed mod 1.  Thus, $\ppq(\zeta) \in \DNQ$.  $\ppq(\zeta) \in \RNQ$ follows directly from \reft{presrat} as $\zeta \in \N{P} \subseteq \RN{P}$.

Let $k$ be a positive integer and suppose that $B$ is a block of length $k$. We note that for large enough $n$, $q_n \leq (\log_2 p_n)^3$, so
$\lim_{n \to \infty} \frac {P_n^{(k)}} {Q_n^{(k)}}=\infty$.  Thus, by \reft{mainpsi}
$$
\lim_{n \to \infty} \frac {N_n^Q\left(\ppq(\eta)\right) } {Q_n^{(k)}}=\lim_{n \to \infty}\left( \frac {N_n^P(\eta) +O(1)} {P_n^{(k)}} \cdot \frac {P_n^{(k)}} {Q_n^{(k)}} \right)=1 \cdot \infty=\infty,
$$
so $\ppq(\zeta) \notin \NQ$.   $Q$ is fully divergent because $\lim_{n \to \infty} \frac {P_n^{(k)}} {Q_n^{(k)}}=\infty$ for all $k$.
\end{proof}

Using \reft{DRdim} it is not difficult to show that   $\dimh{\ppq(\mathbb{R})}=1$.  In fact, we can say even more about $\ppq(\mathbb{R})$. Since $2^t>t^3$ for positive integers $t$ if and only if $t \geq 10$, we can show that the Lebesgue measure of $\ppq(\mathbb{R})$ is positive:
\footnote{This approximation is easily obtained by estimating $\log \lmeas{\ppq(\mathbb{R})}$.}
\begin{align*}
\lmeas{\ppq(\mathbb{R})}
&=\prod_{n=1}^\infty \frac {\min(p_n,q_n)} {q_n}
=\prod_{t=6}^9 \left( \left(  \frac {2^t} {t^3}    \right)^{ |\{n: E_{t,n}=t \hbox{ and } q_{t,n}=t^3\}|} \right)\\
&=\prod_{t=6}^9 \left( \left(  \frac {2^t} {t^3}    \right)^{2^{4t^2}\left(t2^{t^3}-t^22^{t^3-t}\right)} \right)\approx 10^{-1.3095 \times 10^{317}}>0.
\end{align*}
Of course, this number is so small that our approximation doesn't even estimate $\lmeas{\ppq(\mathbb{R})}$ within $10^{310}$ orders of magnitude!


\subsubsection{Further Steps}It should be emphasized that \reft{RDNnonempty}  gives only one example of a basic sequence $Q$ where $\RDN \neq \emptyset$.  
It is likely that $\RDN \neq \emptyset$ for every basic sequence $Q$ that is infinite in limit and fully divergent.  The construction in this section makes heavy use of the number $\zeta$ and estimates pertaining to it from \cite{AlMa} to greatly simplify the proof.  It remains to be seen if the methods introduced in this section generalize well to show that $\RDN$ is always non-empty.  Moreover, it is likely that $\dimh{\RDN}=1$, but it doesn't seem obvious how this would be proven. 

\subsection{The sets $\NQ \backslash \DNQ$, $\RNQ \backslash \NQ$, and $\DNQ \backslash \RNQ$ are always non-empty.}\labs{NRNDN}

In \cite{AlMa}, a computable real number $x$ and a computable basic sequence $Q$ were constructed where $x \in \NQ$, but $T_{Q,n}(x) \to 0$.  
Unfortunately, the approach taken can only be easily extended to a very restrictive class of basic sequences and the proof and construction require some work.  We essentially trivialize the problem of showing that $\NQ \backslash \DNQ \neq \emptyset$ with the theory developed in \refs{ppq}.  The approach used in this subsection is not only simpler, but far stronger than the approach in \cite{AlMa}.  
Examples of computable members of $\DNQ \backslash \RNQ$ are given in \cite{Mance7} for certain classes of computable basic sequences $Q$.
\begin{thrm}\labt{NnotDN}
Suppose that $Q$ is infinite in limit and fully divergent.  Then $\NQ \backslash \DNQ \neq \emptyset$.
\end{thrm}
\begin{proof}
Let $p_n=\max(\floor{\log q_n},2)$ and set $P=(p_n)$.  By the main theorem of \cite{Mance4}, $\NQ \neq \emptyset$, so let $x \in \NQ$ and put
$
y=\left(\ppq \circ \pqp\right) (x).
$
Then $y$ is Q-normal by \reft{mainpsi}, but $T_{Q,n}(y) \to 0$, so $y$ is not $Q$-distribution normal.
\end{proof}
\begin{thrm}\labt{RNnotN}
If $Q$ is infinite in limit, then $\RNQ \backslash \bigcup_{k=1}^{\infty} \Nk{Q}{k} \neq \emptyset$, so $\RNQ \backslash \NQ \neq \emptyset$.
\end{thrm}
\begin{proof}
If $Q$ is $k$-convergent for some $k$, then $\NQ=\emptyset$, but $\RNQ \neq \emptyset$ by Proposition 5.1 and Proposition 5.2 in \cite{Mance4}.  So suppose that $Q$ is fully divergent.  Let $p_n=\max(\floor{q_n/2},2)$ and set $P=(p_n)$.  Clearly, $P$ is fully divergent.  Let $x \in \N{P}$ and set $y=\phpq{k} (x)$.  Let $k$ be a positive integer and suppose that $B_1$ and $B_2$ are blocks of length $k$.  Then by \reft{mainpsi}
$$
\lim_{n \to \infty} \frac {N_n^Q(B_1,y)} {N_n^Q(B_2,y)}
=\lim_{n \to \infty} \frac {N_n^P(B_1,x)+O(1)} {N_n^P(B_2,x)+O(1)}
=\lim_{n \to \infty} \frac {N_n^P(B_1,x)/P_n^{(k)}+o(1)} {N_n^P(B_2,x)/P_n^{(k)}+o(1)}
=1.
$$
Thus, $y \in \RNQ$.  Now, suppose that $B$ is some block of length $k$.  Then, applying \refl{tcorr} by letting $a_j=p_jp_{j+1} \cdots p_{j+k-1}$ and $b_j=q_jq_{j+1} \cdots q_{j+k-1}$
$$
\lim_{n \to \infty} \frac {N_n^Q(B,y)} {Q_n^{(k)}}
=\lim_{n \to \infty} \left( \frac {N_n^Q(B,y)} {P_n^{(k)}} \cdot \frac {P_n^{(k)}} {Q_n^{(k)}} \right)
=\lim_{n \to \infty} \frac {N_n^P(B,x)+O(1)} {P_n^{(k)}} \cdot \lim_{n \to \infty}\frac {P_n^{(k)}} {Q_n^{(k)}}
=1 \cdot 2^{-k} \neq 1.
$$
So, $y \notin \Nk{Q}{k}$ for all $k$.  Thus, $\RNQ \backslash \NQ \neq \emptyset$.
\end{proof}

Using different methods than those used in this paper, it was shown in \cite{Mance7} that $\dimh{\DNQ \backslash \RNQ}=1$. While the methods of this paper appear to be unable to derive that result, we can still provide an alternate proof that $\DNQ \backslash \RNQ \neq \emptyset$.

\begin{thrm}\labt{DNnotRN}
If $Q$ is infinite in limit, then $\DNQ \backslash \RNQ \neq \emptyset$.
\end{thrm}
\begin{proof}
Let $x \in \DNQ$ and set $p_n=q_n-1$.  Put
$
y=\left(\ppq \circ \pqp\right)(x)+\sum_{n=1}^\infty \frac {1} {q_1 \cdots q_n}.
$
Then the digit $0$ never appears in the $Q$-Cantor series expansion of $y$, so $y \notin \RNk{Q}{1} \supseteq \RNQ$.
We note that $\left|T_{Q,n-1}(x)-T_{Q,n-1}(y)\right| \leq \frac {1} {q_n} \to 0$, so the sequence $(T_{Q,n}(y))$ is uniformly distributed mod 1.  Thus, $y \in \DNQ \backslash \RNQ$.
\end{proof}
We will use a pair of basic sequences similar to those from \reft{DNnotRN} in \refs{dimhDN} to sharpen some results on the Hausdorff dimension of $\DNQ \backslash \RNk{Q}{1}$.

\section{The Hausdorff Dimension of some sets}\labs{dimh}

\subsection{Refinement of a result concerning Hausdorff dimension}


For any sequence $X=(x_n)$ of real numbers, let $\mathbb{A}(X)$ denote the set of accumulation points of $X$. 
Given a set $D \subseteq [0,1]$, let
$$
\mathbb{E}_D(Q)=\left\{x=0.E_1E_2\cdots \hbox{ w.r.t. }Q: \mathbb{A}((E_n/q_n))=D      \right\}.
$$
The following results are proven by Y. Wang, Z. Wen, and L. Xi in \cite{WangWenXi}.
\begin{thrm}\labt{WangWenXi}
If $Q$ is infinite in limit, then $\dimh{\mathbb{E}_D(Q)}=1$ for every closed set $D$.
\end{thrm}
\begin{cor}\labc{WangWenXi}
Given $0 \leq \delta \leq 1$, let
$$
\mathbb{E}_\delta(Q)=\mathbb{E}_{\{\delta\}}(Q)=\left\{ x=0.E_1E_2\cdots \hbox{ w.r.t. }Q: \lim_{n \to \infty} \frac {E_n} {q_n}=\delta     \right\}.
$$
If $Q$ is infinite in limit, then $\dimh{\mathbb{E}_\delta(Q)}=1$.
\end{cor}

For a set $D \subseteq [0,1]$ and sequence of non-negative integers $(t_n)$, let
\begin{align*}
\mathbb{E}_{D,(t_n)}(Q)&=\mathbb{E}_D(Q) \cap \mathscr{R}_{(q_n-t_n)}(Q)\\
&=\left\{x=0.E_1E_2\cdots \hbox{ w.r.t. }Q: \mathbb{A}((E_n/q_n))=D \hbox{ and }\forall n \ E_n<q_n-t_n       \right\}.
\end{align*}

\begin{lem}\labl{logprod}
If $\lim_{n \to \infty} \frac {\log q_{n+j+1}} {\log \q{n}}=0$ for all $j\geq 0$, then $\lim_{n \to \infty} \frac {\log q_{n+1}  \cdots q_{n+k+1}} {\log \q{n}}=0$ for all $k \geq 0$.
\end{lem}
\begin{proof}
This follows immediately as $\frac {\log q_{n+1}  \cdots q_{n+k+1}} {\log \q{n}}=\frac {\log q_{n+1}} {\log \q{n}}+\cdots+\frac {\log q_{n+k+1}} {\log \q{n}}$.
\end{proof}

\begin{lem}\labl{phqpHolder}
If $Q$ is infinite in limit, $q_n \geq p_n, \lim_{n \to \infty} \frac {\log q_{n+j+1}} {\log \q{n}}=0$ for all $j\geq 0,\  \sum_{n=1}^\infty \frac {q_n-p_n} {q_n}<\infty$, and $\min(p_n,q_n) \geq 3$ for all $n$, then for all $k \geq 0, \left(\phpq{k}\right)^{-1}=\phqp{k}$ is H\"older continuous of exponent $\al$ for all $\al \in (0,1)$.
\end{lem}
\begin{proof}
Let $\al \in (0,1)$ and $k\geq 0$.   First, we note that
\begin{equation}\labeq{Holder3}
\prod_{j=1}^\infty \frac {p_j} {q_j}=\prod_{j=1}^{\infty} \frac {q_j-(q_j-p_j)} {q_j}
=\prod_{j=1}^{\infty} \left(1-\frac {q_j-p_j} {q_j}\right)>0,
\end{equation}
as $\sum \frac {q_n-p_n} {q_n}<\infty$.  
Since $\lim_{n \to \infty} \frac {\log q_{n+1}\cdots q_{n+k+1}} {\log \q{n}}=0$ by \refl{logprod} and $\lim_{n \to \infty} \frac {q_n-p_n} {q_n}=0$, we know that $\lim_{n \to \infty} \frac {\log q_{n+1}\cdots q_{n+k+1}} {\log \p{n}}=0$, so
\begin{equation}\labeq{Holder4}
\lim_{n \to \infty} \frac {q_{n+1}\cdots q_{n+k+1}} {(\p{n})^{\frac {1-\al} {\al}}}=0.
\end{equation}
To verify \refeq{Holder2}
\begin{align*}
&\limsup_{n \to \infty} \frac {(\q{n})^\alpha}  {\p{n}} \cdot \left( \frac {q_{n+1}\cdots q_{n+k+1}} {\max(1,q_{n+1}-p_{n+1})}\right)^\al 
\leq \limsup_{n \to \infty}  \frac {(\q{n+k+1})^\al} {\p{n}}\\
=& \limsup_{n \to \infty} \left(\prod_{j=1}^n \frac {q_j} {p_j}\right)^\al \cdot \left(\frac {q_{n+1}\cdots q_{n+k+1}} {(\p{n})^{\frac {1-\al} {\al}}}\right)^\al=\left(\prod_{j=1}^\infty \frac {q_j} {p_j}\right)^\al \cdot 0=0,
\end{align*}
by \refeq{Holder3} and \refeq{Holder4}, verifying \refeq{Holder2}.  We can use similar methods to prove that $$\limsup_{n \to \infty} \frac {(\q{n})^\al} {\p{n}} \cdot q_{n+1}^{1-\al}<~\infty,$$ verifying \refeq{Holder1}.
\end{proof}

We prove the following refinement of \reft{WangWenXi}:
\begin{thrm}\labt{refinedWWX}
Suppose that $D\subseteq (0,1)$ is a closed set and $(t_n)$ is a sequence of non-negative integers.  If $Q$ is infinte in limit, $\lim_{n \to \infty} \frac {\log q_{n+j}} {\log \q{n}}=0$ for all $j \in \mathbb{N}$,  $\sum_{n=1}^\infty \frac {t_n} {q_n}<\infty$, and $q_n-t_n \geq 3$ for all $n$, then $\dimh{\mathbb{E}_{D,(t_n)}(Q)}=1$.
\end{thrm}
\begin{proof}
Let $p_n=q_n-t_n$.  We will show that
\begin{equation}\labeq{refinedWWXgoal}
\mathbb{E}_{D,(t_n)}(Q)=\ppq\left(\mathbb{E}_D(P)\right).
\end{equation}
Let $x=\formalsum \in \ppq\left(\mathbb{E}_D(P)\right)$ and $y \in D$.  Thus, for all $\epsilon>0$, there exists $n$ such that $\left| \frac {E_n} {q_n-t_n}-y \right|<~\epsilon$.  Note that
$$
\left| \frac {E_n} {q_n}-y \right| \leq \left| \frac {E_n} {q_n}-\frac {E_n} {q_n-t_n} \right|+\left|\frac {E_n} {q_n-t_n}-y\right|
<\frac {E_n t_n} {q_n(q_n-t_n)}+\epsilon<\frac {t_n} {q_n}+\epsilon.
$$
Since $\sum t_n/q_n<\infty$, we know that $\frac {t_n}{q_n} \to 0$, so $x \in \mathbb{E}_{D,(t_n)}(Q)$.  The proof that $\mathbb{E}_{D,(t_n)}(Q) \subseteq \ppq\left(\mathbb{E}_D(P)\right)$ is similar, so \refeq{refinedWWXgoal} holds.

We note that $\EDP \subseteq \bigcup_{k=0}^\infty \zpq{k}$ since $\rho_{P}(x)<\infty$ for all $x \in \EDP$ as $0$ and $1$ are not members of $D$.  Similarly, $\EDQ \subseteq \bigcup_{k=0}^\infty \zqp{k}$.  Put $A_k=\EDP \cap \zpq{k}$ and $B_k=\EDQ \cap \zqp{k}$, so that $\EDP=\bigcup_{k=0}^\infty A_k$ and $\EDQ=\bigcup_{k=0}^\infty B_k$.  Thus,
\begin{equation}\labeq{Holder5}
\dimh{\EDP}=\sup \dimh{A_k} \hbox{ and } \dimh{\EDQ}=\sup \dimh{B_k}.
\end{equation}
By \reft{WangWenXi} and \refeq{Holder5}, $\sup \dimh{A_k}=1$.  Let $k \geq 0$.
Next, we note that $\phqp{k}(B_k)=A_k$ by \refeq{refinedWWXgoal}.  Thus, by \refl{phqpHolder}, $\dimh{A_k} \leq \frac {1} {\al} \dimh{B_k}$ for all $\al \in (0,1)$, so $\dimh{B_k} \geq \dimh{A_k}$.  But then $\sup \dimh{B_k} \geq \sup \dimh{A_k}=1$, so $\sup \dimh{B_k}=1$.  Thus, $\dimh{\EDQ}=1$ by \refeq{Holder5}.
\end{proof}

\subsection{The Hausdorff dimension of $(\DN{Q} \backslash \RN{Q}) \cap \mathscr{R}_{(t_n)}$}\labs{dimhDN}

\begin{definition}
Let $P=(p_n)$ and $Q=(q_n)$ be basic sequences.  We say that $P \sim_s Q$ if
$
q_n=\prod_{j=1}^s p_{s(n-1)+j}.
$
\end{definition}

The following theorem  was proven in \cite{Mance7}.

\begin{thrm}\labt{mance7thm}
Suppose that $(Q_j)_{j=1}^\infty$ is a sequence of basic sequences that are infinite in limit.  Then
$$
\dimh{\RNisect}=1
$$
if either
\begin{enumerate}
\item $Q_j$ is $1$-convergent for all $j$ or
\item $Q_1$ is $1$-divergent and there exists some basic sequence $S=(s_n)$ with $Q_1 \LS{1} Q_2 \LS{2} Q_3 \LS{3} Q_4 \cdots.$
\end{enumerate}
\end{thrm}

The following may be proven similarly to \reft{refinedWWX}.

\begin{thrm}
Suppose that $(t_n)$ is a sequence of non-negative integers, $(Q_j)_{j=1}^\infty$ is a sequence of basic sequences that are infinite in limit, $Q_1=(q_n)$ is $1$-divergent, there exists some basic sequence $S=(s_n)$ with
$
Q_1 \LS{1} Q_2 \LS{2} Q_3 \LS{3} Q_4 \cdots,
$
$\sum_{n=1}^\infty \frac {t_n} {q_n}=0$, and $q_n-t_n \geq 3$ for all $n$.  Then
$$
\dimh{\mathscr{R}_{(q_n-t_n)} \cap \RNisect}=1.
$$
\end{thrm}

\subsection{The sets $\zpq{k}$}

It seems to be difficult to compute the exact Hausdorff dimension of any of the sets $\zpq{k}, \bigcup_k \zpq{k}$, or $\zpq{\infty}$.  It is likely that an extension of \cite{UrbanskiRempeGillen} would provide a solution to this problem, but this is beyond the scope of the current paper.  However, the following is easily seen to follow from \reft{wegmannsalat}.

\begin{thrm}  Suppose that $\lim_{n \to \infty} \frac {\log p_n} {\log \p{n}}=0$.  Then for $k \geq 0$
\begin{align*}
\dimh{\zpq{\infty}} &= \liminf_{n \to \infty} \frac {\log \prod_{j=1}^n \min(p_j,q_j)} {\log \prod_{j=1}^n p_j} \geq \sup_j \dimh{\zpq{j}} = \dimh{\bigcup_{j=0}^\infty \zpq{j}}\\
& \geq \dimh{\zpq{k}} \geq \dimh{\zpq{0}} = \liminf_{n \to \infty} \frac {\log \prod_{j=1}^n \min(p_j-2,q_j-1)} {\log \prod_{j=1}^n p_j}.
\end{align*}
\end{thrm}


\begin{thrm}
Suppose that $\mu_1, \mu_2 \in \measNN$ and  $\mu_1$ and $\mu_2$ are not positive on $\{2\}$.  Put $\mu=\muonetwo$ and suppose that $\max\left(\ELone,\ELtwo\right)<\infty$.  Then for all $k \in \mathbb{N}_0$ and $\mu$-almost every $\PQ \in \NNC$
$$
\frac {\int \log \min(\pione{\omega}-2,\pitwo{\omega}-1) \ d\mu(\omega)} {\ELone}
\leq \dimh{\zpq{k}} \leq
\frac {\int \log \min(\pione{\omega},\pitwo{\omega}) \ d\mu(\omega)} {\ELone}.
$$
\end{thrm}
\bibliographystyle{plain}


\end{document}